\documentclass{amsart}

\usepackage{enumitem}

\theoremstyle{plain}
\newtheorem{theorem}{Theorem}[section]
\newtheorem{proposition}[theorem]{Proposition}
\newtheorem{lemma}[theorem]{Lemma}

\newtheorem*{theorem*}{Theorem}

\theoremstyle{definition}
\newtheorem{definition}[theorem]{Definition}
\newtheorem*{definition*}{Definition}

\theoremstyle{remark}
\newtheorem{remark}{Remark}[section]

\newtheorem{question}{Question}

\DeclareMathOperator{\PSL}{PSL}
\DeclareMathOperator{\GL}{GL}
\title[Small doubling for discrete subsets]{Small doubling for discrete subsets of non-commutative groups and a theorem of Lagarias}

\author{Simon Machado}
\address{Institute for Advanced Study\\
 1 Einstein drive\\
  08540 Princeton, NJ}
\email{machado@ias.edu}
\begin{document}
\begin{abstract}
Approximate lattices of Euclidean spaces, also known as Meyer sets, are aperiodic subsets with fascinating properties. In general, approximate lattices are defined as approximate subgroups of locally compact groups that are discrete and have finite co-volume. A theorem of Lagarias provides a criterion for discrete subsets of Euclidean spaces to be approximate lattices. It asserts that if a subset $X$ of $\mathbb{R}^n$ is relatively dense and $X - X$ is uniformly discrete, then $X$ is an approximate lattice. 

We prove two generalisation of Lagarias' theorem: when the ambient group is amenable and when it is a higher-rank simple algebraic group over a characteristic $0$ local field. This is a natural counterpart to the recent structure results for approximate lattices in non-commutative locally compact groups. We also provide a reformulation in dynamical terms pertaining to return times of cross-sections.  Our method relies on counting arguments involving the so-called periodization maps, ergodic theorems and a method of Tao regarding small doubling for finite subsets. In the case of simple algebraic groups over local fields, we moreover make use of deep superrigidity results due to Margulis and to Zimmer. 
\end{abstract}

\maketitle
\tableofcontents
\section{Introduction}
A subset $\Lambda$ of a locally compact group $G$ is called an \emph{approximate lattice} (\cite{bjorklund2016approximate, hrushovski2020beyond}) if it satisfies:
\begin{itemize}
\item (uniformly discrete) there is a neighbourhood $W$ of the identity such that 
$$ \forall \lambda_1 \neq \lambda_2 \in \Lambda, \lambda_1W \cap \lambda_2W = \emptyset;$$
\item (finite co-volume) there is $\mathcal{F}\subset G$ with finite Haar measure such that $$\Lambda\mathcal{F}:=\{\lambda f : \lambda \in \Lambda, f \in \mathcal{F}\}=G;$$
\item ($l$-approximate subgroup) $e \in \Lambda$, $\Lambda=\Lambda^{-1}$ and there exists $F \subset G$ of size at most $l$ such that 
$$ \Lambda^2:=\{\lambda_1\lambda_2 : \lambda_1,\lambda_2 \in \Lambda\} \subset F\Lambda:=\{f\lambda : f \in F, \lambda \in \Lambda\}.$$
\end{itemize}

Ever since their introduction by Meyer in his seminal monograph \cite{meyer1972algebraic}, approximate lattices of Euclidean spaces - and their more regular relatives the model sets - have been key objects of study in the theory of aperiodic order (see for instance \cite{meyer1972algebraic, moody1997meyer, MR1400744}). Their definition allowed to study in a common framework Penrose tilings (and their pentagrid description due de Bruijn \cite{deBruijn1981algebraicII, deBruijn1981algebraicI}), Pisot--Vijayaraghavan numbers \cite{meyer1972algebraic}, crystallographic subsets and certain mathematical models of quasi-crystals \cite{MR3136260} or spectra of certain crystalline measures related to Poisson summation (\cite{NirOlevskii2015PoissonSummation}). More recently, inspired by the breakthroughs in the study of non-commutative finite approximate subgroups (e.g. \cite{MR2415382, MR2833482, MR3090256}), the interest has grown concerning non-commutative approximate lattices. On the one hand, generalisations of Meyer's structure theorem \cite{meyer1972algebraic} have been extensively studied. This study was initiated by Bj\"{o}rklund--Hartnick in \cite{bjorklund2016approximate} and generalisations to linear groups have been achieved through, notably, breakthrough work of Hrushovski \cite{hrushovski2020beyond} and work of the author \cite{machado2020approximate, machado2020apphigherrank, machado2022discrete}. On the other hand, approximate lattices have been used as instances of structured subsets with rich aperiodic behaviour in non-commutative frameworks (see for instance the monograph \cite{cordes2020foundations} and recent works \cite{EnstadVelthoven, https://doi.org/10.48550/arxiv.2204.01496, HeuerKielak, PRS2022LeptinDensities, enstad2022dynamical}). 

A result of Lagarias' \cite{MR1400744} fits naturally within that picture. It provides a geometrically significant condition for a discrete co-compact subset of a Euclidean space to be an approximate lattice (see \cite[\S 5.2]{NirOlevskii2015PoissonSummation} for a short proof and a deep application). Lagarias' theorem asserts that if a subset $X \subset \mathbb{R}^n$ is such that $X + K = \mathbb{R}^n$ for some compact subset $K$ and $X-X$ is uniformly discrete, then $X$ is an approximate lattice. Our goal in this article is to provide generalisations of Lagarias' theorem that, in particular, hold in the non-commutative framework.

Our first theorem generalises Lagarias' theorem to amenable groups: 
 
 \begin{theorem}\label{Theorem: Lagarias-type theorem in amenable groups, first version}
Let $G$ be a unimodular amenable second countable locally compact group. Let $X \subset G$ be such that: 
\begin{enumerate}
\item $X^{-1}X$ is uniformly discrete i.e. $(X^{-1}X)^2 \cap V^{-1}V=\{e\}$ for some neighbourhood of the identity $V \subset G$;
\item there is $\mathcal{F}$ Borel of finite Haar measure such that $\mathcal{F}X=G$. 
\end{enumerate}
Then there are an approximate lattice $\Lambda \subset XX^{-1}$ and a finite subset $F \subset G$ such that $X \subset \Lambda F$. In addition, $F$ and $\Lambda$ can be chosen so $|F| = O((\mu_G(\mathcal{F})/\mu_G(V))^3)$ and $\Lambda$ is an $O(\mu_G(\mathcal{F})/\mu_G(V)^{12})$-approximate subgroup whose internal dimension (see \S \ref{Internal dimension of a model set}) is bounded above by $O(\log(\mu_G(\mathcal{F})/\mu_G(V)))$.
\end{theorem}

Qualitatively, Lagarias' theorem and Theorem \ref{Theorem: Lagarias-type theorem in amenable groups, first version} are instances of the following heuristic: ``small doubling'' (i.e. $X^{-1}X$ not much larger than $X$) implies ``approximate subgroup''. Here, if smallness is determined by how discrete a subset is, then Theorem \ref{Theorem: Lagarias-type theorem in amenable groups, first version} indeed tells us that a finite co-volume subset with small doubling is an approximate subgroup. Such heuristics are commonplace in additive combinatorics. The Pl\"{u}nnecke--Ruzsa theorem relates \emph{quantitatively} small doubling for finite subsets - in terms of size - and approximate subgroups (see \cite{MR3063158, MR1701207}, a non-commutative generalisation due to Tao \cite{MR2501249}, as well as the flattening lemma of Bourgain--Gamburd \cite{MR2415383}). This analogy is anything but vain  and, in fact, lies at the heart of our approach. Konieczny \cite{konieczny2021characterisation} was the first to use this analogy to give a slick new proof, and indeed a quantitative improvement, of Lagarias' theorem in the case $G=\mathbb{R}^n$. The methods concerned with discrete subsets of Euclidean spaces do not generalise to amenable ambient groups however, see Section \ref{Subsection: Proof strategy} for a description of where it fails and how we circumvent this issue. 

Our method shows furthermore that \emph{finite co-volume} is better embodied by the following notion. Given $X \subset G$ closed, let $\Omega_X$ denote the closure - in the Chabauty space $\mathcal{C}(G)$ of closed subsets of $G$ equipped with the Chabauty topology (Subsection \ref{Subsection: The Invariant hull}) - of the orbit $G\cdot X$. We call $\Omega_X$ the \emph{invariant hull} of $X$. Sometimes also called the tiling space or the continuous hull of $X$ \cite{MR3136260}, $\Omega_X$ is a key object in the theory of aperiodic order that encapsulates the long-range structure of the (potentially) aperiodic subset $X$ of $G$. We are able to prove the following generalisation of Theorem \ref{Theorem: Lagarias-type theorem in amenable groups, first version}:

\begin{theorem}\label{Theorem: Lagarias-type theorem in amenable groups, hull version}
Let $G$ be a unimodular amenable second countable locally compact group. Let $X \subset G$ be such that: 
\begin{enumerate}
\item $X^{-1}X$ is uniformly discrete i.e. $(X^{-1}X)^2 \cap V^{-1}V=\{e\}$ for some neighbourhood of the identity $V \subset G$;
\item there is a proper $G$-invariant Borel probability measure $\nu$ on $\Omega_X$. 
\end{enumerate}
Then for $\nu$-almost all $Y \in \Omega_X$, there are an approximate lattice $\Lambda \subset YY^{-1}$ and a finite subset $F \subset G$ such that $Y \subset \Lambda F$. In addition, if $\mu_G$ denotes the Haar measure $\mathcal{P}_X^*\nu$ (Subsection \ref{Subsection: The Invariant hull}), then $F$ and $\Lambda$ can be chosen so $|F| = O(\mu_G(V)^{-3})$, $\Lambda$ is an $O(\mu_G(V)^{-12})$-approximate subgroup and $\Lambda$ has internal dimension $O(\log \mu_G(V)^{-1})$.
\end{theorem}

Even in the case of an abelian ambient group, our results extend those of Lagarias \cite{MR1400744} as they apply more generally to subsets of finite co-volume rather than relatively dense subsets alone. Furthermore, examples of subsets satisfying the conditions of Theorem \ref{Theorem: Lagarias-type theorem in amenable groups, hull version} but not those of Theorem \ref{Theorem: Lagarias-type theorem in amenable groups, first version} are plenty. A natural example of number-theoretic origin is the subset of $\mathbb{Z}^2$ made of pairs of co-prime integers - its co-volume is $\frac{6}{\pi^2}$ (note that, in this particular case, Theorem \ref{Theorem: Lagarias-type theorem in amenable groups, hull version} can also be established by direct means).

We also take our investigation beyond the realm of amenable locally compact groups. We prove that a Lagarias-type theorem also holds in simple Lie groups, whose behaviour is the polar opposite of that of amenable groups. 

\begin{theorem}\label{Theorem: Lagarias-type theorem in simple groups}
Let $k$ be a local field of characteristic $0$. Let $G$ be the group of $k$-points of a simple algebraic group $\mathbb{G}$ defined over $k$ of $k$-rank at least $2$. Let $X \subset G$ be such that: 
\begin{enumerate}
\item $XX^{-1}$ and $X^{-1}X$ are uniformly discrete;
\item there is a proper $G$-invariant Borel probability measure on $\Omega_X$. 
\end{enumerate}
Then there is an approximate lattice $\Lambda \subset (X^{-1}X)^2$ and a finite subset $F \subset G$ such that $X \subset F\Lambda$.
\end{theorem}

An interesting example one may keep in mind is the case $G=\PSL_n(\mathbb{R}), n \geq 3$ as this case already contains all the difficulty of the proof. The proof of Theorem \ref{Theorem: Lagarias-type theorem in simple groups} shows that the subsets studied are extremely rigid and exhibit even more striking behaviour than in the amenable world. This reminisces works of Margulis--Mozes \cite{zbMATH01239670}, Mozes \cite{MR1452434} and Block--Weinberger \cite{zbMATH00109492} who take advantage of rigidity properties in symmetric spaces and non-amenable spaces to provide beautiful examples of aperiodic tilings. Theorem \ref{Theorem: Lagarias-type theorem in simple groups} also provides new information regarding the structure of approximate lattices in semi-simple groups and is an interesting counterpart to Hrushovski's \cite{hrushovski2020beyond}. Compared to \cite{hrushovski2020beyond} we are able to prove a structure theorem while assuming that a stronger ``finite co-volume'' assumption holds \emph{but}, as a trade-off,  we do not assume that $X$ is an approximate subgroup and assume only a weaker algebraic relation between $XX^{-1}$, $X^{-1}X$ and $X$. 

\subsection{A dynamical reformulation and two questions}
The proof of Theorem \ref{Theorem: Lagarias-type theorem in simple groups} relies greatly on considering the restriction of the dynamics on the invariant hull to a naturally defined cross-section. Given a dynamical system $(X, \nu)$ equipped with a probability measure preserving ergodic action of a locally group $G$, we say that a Borel subset $B \subset X$ is a \emph{cross-section} if for some neighbourhood of the identity $W \subset G$ the map $W \times B \rightarrow X$ defined by $(w,b) \mapsto wb$ is one-to-one and $\nu(GB)=1$. For $x \in X$, we define the set of \emph{hitting times} $T_h(x,B)$ as the subset $\{g \in G : gx \in B\}$. Since $B$ is a cross-section, the subset $T_h(x,B)$ is uniformly discrete for all $B$. 

It was shown by Bj\"{o}rklund, Hartnick and Karasik in \cite{https://doi.org/10.48550/arxiv.2108.09064}, that the map 
\begin{align*}
\phi: X &\longrightarrow \mathcal{C}(G)\\
x &\longmapsto T_h(x,B)^{-1}
\end{align*}
is a Borel $G$-equivariant map. There is moreover $X_0 \subset \mathcal{C}(G)$ such that $\phi$ takes values in $\Omega_{X_0}$ for $\nu$-almost all $x \in X$. Set $\nu_0$ the push-forward of $\nu$ via $\phi$. Since cross-sections exist as soon as $G$ acts freely, this produces an abundance of uniformly discrete subsets $X_0$ for which $\Omega_{X_0}$ admits a proper $G$-invariant Borel probability measure, see \cite{https://doi.org/10.48550/arxiv.2108.09064} and references therein. 

Using this language, Theorem \ref{Theorem: Lagarias-type theorem in amenable groups, hull version} may be rephrased as:
 
\begin{theorem}
Let $G$ be a unimodular amenable second countable locally compact group acting by probability measure preserving action on $(X, \nu)$. Suppose that $B \subset G$ is a cross-section of $X$ and that the set of \emph{return times} of $B$
$$\mathcal{R}(B):=\{g \in G: gB \cap B \neq \emptyset\} = \bigcup_{y \in B} T_{hit}(y,B)$$ is uniformly discrete. Then there is an approximate lattice $\Lambda \subset G$ such that $T_h(x,B)$ is covered by finitely many translates of $\Lambda$ for $\nu$-almost every $x \in X$. 
\end{theorem}

It would be interesting to be able to extend this theorem to the case of a simple Lie group acting on $X$. However, Theorem \ref{Theorem: Lagarias-type theorem in simple groups} does not allow us to do so because condition (1) requires both $XX^{-1}$ and $X^{-1}X$ to be uniformly discrete. We therefore ask:

\begin{question}\label{Question: First}
Let $k$ be a local field of characteristic $0$. Let $G$ be the group of $k$-points of a simple algebraic group $\mathbb{G}$ defined over $k$ of $k$-rank at least $2$. Let $X$ be such that: \begin{enumerate}
\item $X^{-1}X$ is uniformly discrete;
\item there is $\mathcal{F}$ Borel of finite Haar measure such that $X\mathcal{F}=G$. 
\end{enumerate} Is there an approximate lattice $\Lambda \subset \langle X \rangle$ and a finite subset $F \subset G$ such that $X \subset F \Lambda$? 
\end{question}

\noindent It is not clear to us how one could approach such a question, and whether the tools from either \cite{hrushovski2020beyond} or this paper (see also \cite{machado2020apphigherrank}) can be applied. 

Furthermore, as is apparent in the proof of Theorem \ref{Theorem: Lagarias-type theorem in simple groups}, the techniques we rely on are infinitary by nature. We are not able there to go from the local structure of $\Lambda$ to its global structure and we rely in particular on ideas related to Margulis' arithmeticity and superrigidity theorems. That is why we are not able to exhibit the same quantitative conclusions in the simple case (Theorem \ref{Theorem: Lagarias-type theorem in simple groups}) which provides a second difference with the case of an amenable group acting (Theorem \ref{Theorem: Lagarias-type theorem in amenable groups, hull version}). This prompts us to ask: 

\begin{question}\label{Question: Second}
Let $F$ be as in the conclusion of Theorem \ref{Theorem: Lagarias-type theorem in simple groups}. Can $|F|$ be bounded above by a quantity depending on how uniformly discrete $X^{-1}X$ is and the co-volume of $X$ (see Lemma \ref{Lemma: finite co-volume passes to commensurability} for a guess at what the co-volume might be)? 
\end{question}
\noindent It is the guess of the author that answering Question \ref{Question: First} would eventually lead to an answer to Question \ref{Question: Second}.

\subsection{Proof strategy}\label{Subsection: Proof strategy}  In Section \ref{Section: Counting points using periodization maps and consequences} we exploit invariant measures on $\Omega_X$ to draw information on finite configurations present in $X$. Precisely, we explain the link between certain maps on the invariant hull - the so-called \emph{periodization maps} inspired by the Siegel transform \cite{10.2307/1969027} (and first used in our context in \cite{bjorklund2016approximate}) - and quantities such as $|Y \cap K|$ for a given compact subset $K$ and a generic $Y \in \Omega_X$. The main takeaway from this article is that periodization maps enable us - in many different ways - to relate ergodic-theoretic information with counting arguments.

In Section \ref{Section: Lagarias-type result in amenable groups}, we prove Theorems \ref{Theorem: Lagarias-type theorem in amenable groups, first version} and \ref{Theorem: Lagarias-type theorem in amenable groups, hull version}. The proof stems from an analogy with finite approximate subgroups. Konieczny was the first to show, in \cite{konieczny2021characterisation}, that this comparison can be used effectively when $G = \mathbb{R}^n$.  There he noticed that the subsets studied by Lagarias can be recovered as a union of subsets with small doubling in the sense of additive combinatorics simply by considering intersections with boxes of growing size.  In general amenable groups this method fails (cf Remark \ref{Remark: Boxes and nilpotency}). Indeed, in a general amenable group $G$ there is no sequence of $K$-approximate subgroups $(F_n)_{n \geq 0}$ (for some fixed $K$) whose union is $G$, see \cite[\S 5]{MR1865397} for a case where an even stronger property fails, as well as the generalisation of Gromov's polynomial growth theorem \cite[Thm. 7.1]{MR2833482}. Nonetheless, if $(F_n)_{n \geq 0}$ denotes a F\o lner sequence, then the intersections $(X \cap F_n)_{n \geq 0}$ satisfy additive combinatorial conditions \emph{as a family} as soon as they are sufficiently large - even though each individual set $X \cap F_n$ might be far from enjoying small doubling properties. We then draw estimates on the size of $X \cap F_n$ from applications of ergodic theorems to periodization maps - and more precisely to  $Y \in \Omega_X\mapsto |Y \cap K|$ for $K$ compact.

The proof strategy for Theorem \ref{Theorem: Lagarias-type theorem in simple groups} resembles the approach of Theorem \ref{Theorem: Lagarias-type theorem in amenable groups, hull version}. However, since the additive-combinatorial tools from \cite{MR2501249} mostly fail in simple Lie groups - due in part to the absence of F\o lner sets - we resort to utilising much more powerful ergodic-theoretic results pertaining to the rigidity of measure-preserving actions of simple Lie groups. The main step of our approach consists in extending Margulis' superrigidity theorem to certain large subsets of $X$. It builds upon previous results of the author established in his proof of a Meyer-type theorem in higher-rank semi-simple groups \cite{machado2020apphigherrank} along with a crucial theory of transverse subsets in invariant hulls studied by Bj\"{o}rklund, Hartnick and Karasik in \cite{https://doi.org/10.48550/arxiv.2108.09064}. Even though, amidst the most technical parts of this argument we make a crucial use of simple and beautiful ideas from combinatorics such as covering lemmas and Massicot--Wagner-type arguments (see the proof of Theorem \ref{Theorem: Superrigidity for return times, version 2}). We introduce the building blocks of that approach in Section \ref{Section: The invariant hull, equivariant families and cross-sections} and, then, prove Theorem \ref{Theorem: Lagarias-type theorem in simple groups} in Section \ref{Section: Lagarias-type theorem in simple Lie groups}. Compared with the approach in \cite[\S 5]{machado2020apphigherrank}, we make sure to solely use counting arguments rather than stronger structural results regarding infinite approximate subgroups.

\section{Background material and general properties}
\subsection{Notations}

Given two subsets $X,Y$ of a group $G$ we will denote
$XY:=\{xy \in G : x \in X, y \in Y\}$, $X^{-1}:=\{x^{-1} \in G : x \in X\}$ and $X^n:=\{x_1 \cdots x_n \in G : x_1,\ldots, x_n  \in X\}$. Write $\langle X \rangle$ the group generated by $X$.

If $G$ is a locally compact group, $\mu_G$ will denote a Haar measure. Let $X$ be a compact space equipped with a continuous action of some locally compact group $G$. Given two finite Borel measures $\mu$ and $\nu$ on $G$ and $X$ respectively, we define the convolution $\mu * \nu$ as 
$$\mu * \nu (\phi):= \int_{G \times X}\phi(gx) d\mu(g) d\nu(x)$$
for all $\phi \in C^0_c(X).$

\subsection{Bi-F\o lner sequences}
Take $G$ a locally compact group and $\mu_G$ a Haar measure. A F\o lner sequence  of a second countable locally compact group $G$ is an increasing sequence $(F_n)_{n \geq 0}$ of compact subsets such that for any other compact subset $K$ and any $\epsilon > 0$ there is $n_0 \geq 0$ with:

$$ \frac{\mu_G(F_n \Delta (kF_n))}{\mu_G(F_n)} \leq \epsilon$$
for all $n \geq n_0$ and $k \in K$. 

It is well-known that a locally compact group admits a F\o lner sequence if and only if $G$ is amenable (see \cite{MR0251549} for a general introduction to amenable groups). In what follows, we will use the fact that one may in fact find F\o lner sequences satisfying stronger assumptions in unimodular amenable second countable locally compact groups. Indeed, Ornstein and Weiss proved:

\begin{proposition}[\cite{MR910005}]
Let $G$ be a unimodular amenable second countable locally compact group. There is a sequence of symmetric compact subsets $(F_n)_{n \geq 0}$ such that for any $K \subset G$ compact we have:
$$ \lim \frac{\mu_G((KF_nK) \Delta F_n)}{\mu_G(F_n)} = 0.$$
We call such a sequence a \emph{bi-F\o lner} sequence. 
\end{proposition}

\subsection{Elementary tools from additive combinatorics}
The combinatorial tools we will use in what follows are mostly elementary. A key role will be played by covering lemmas. 

\begin{lemma}[Ruzsa's covering lemma, \cite{MR1200845}]
Let $X,Y$ be two subsets of a group $G$. Suppose that $F \subset Y$ is maximal such that the subsets $fX$ for $f\in F$ are pairwise disjoint. Then 
$$ Y \subset FXX^{-1}.$$
In particular, if there is a bound $C>0$ on the size of subsets $F'$ of $Y$ such that the subsets $fX$ for $f \in F'$ are pairwise disjoint, then 
$$ Y \subset F''XX^{-1}$$ 
for some subset $F''$ of size at most $C$. 
\end{lemma}

We will also use a straightforward generalisation of this lemma. 

\begin{lemma}[Ruzsa's covering lemma for families of sets]\label{Lemma: Ruzsa's covering lemma for families of sets}
Let $(X_i)_{i \in I}$ be  a family of subsets of $G$ and let $Y$ be a subset of a group $G$. Suppose that there is a bound $C>0$ on the size of subsets $F'$ of $X$ such that for all $i \in I$ the subsets $fX_i$ for $f \in F'$ are pairwise disjoint, then 
$$ Y \subset F''\left(\bigcup_{i \in I} X_iX_i^{-1}\right)$$ 
for some subset $F''$ of size at most $C$. 
\end{lemma}

\begin{proof}
Let $F''$ be a subset of maximal cardinality satisfying the assumption. Then, by maximality, for all $y \in Y$ there are $i \in I$ and $f\in F''$ such that $fX_i \cap yX_i \neq \emptyset$ i.e. $y \in fX_iX_i^{-1}$. Since $F''$ has size at most $C$, this concludes the proof. 
\end{proof}

\subsection{Internal dimension of a model set}\label{Internal dimension of a model set}

Suppose that $\Lambda$ is commensurable with a model set associated to a cut-and-project scheme $(G,H,\Gamma)$ (see \cite[\S 2.3]{bjorklund2016approximate}). By the Gleason--Yamabe theorem \cite{10.2307/1969792}, there is an open subgroup $U \subset H$ and a compact normal subgroup $K \subset U$ such that $U/K$ is a connected Lie group without non-trivial compact normal subgroup. The subgroup thus obtained is essentially unique i.e. it depends on the commensurability class of $\Lambda$ only (see \cite[Proposition 3.6,(2)]{machado2019goodmodels}). Equivalently, it is the connected Lie group $H'$ of minimal dimension such that there exists a cut-and-project scheme $(G,H',\Gamma')$ whose model sets are commensurable with $\Lambda$. We thus define the \emph{internal dimension} of $\Lambda$ as the dimension of $U/K$. A recent result of An--Jing--Tran--Zhang relying on the nonabelian Brunn--Minkowski inequality \cite{jing2021nonabelian} implies:

\begin{theorem}[Theorem 1, \cite{an2021small}]
If $\Lambda$ is a $K$ approximate subgroup and commensurable to a model set, then $\Lambda$ has internal dimension at most $O\left( \log_2^2 K\right)$.
\end{theorem}

Building upon that result as well as a structure theorem for amenable approximate subgroups we proved furthermore:

\begin{proposition}[Proof of Theorem 1.8 (2), \cite{machado2019goodmodels}]
If $\Lambda$ is an approximate lattice and a $K$-approximate subgroup of some second countable locally compact group, then $\Lambda$ has internal dimension at most $18\log_2(K)$.
\end{proposition}

\section{Counting points using periodization maps and consequences}\label{Section: Counting points using periodization maps and consequences}
\subsection{The invariant hull}\label{Subsection: The Invariant hull}

Let $G$ be a second countable locally compact group. The Chabauty space of $G$ is the set $\mathcal{C}(G)$ of all closed subsets of $G$ - possibly empty - equipped with the topology generated by the open subsets 
$$U^V:=\{X \in \mathcal{C}(G) : V \cap X \neq \emptyset\}$$
and 
$$U_K:=\{X \in \mathcal{C}(G) : K \cap X = \emptyset\}$$
for all $V$ open and $K$ compact. The Chabauty space is a compact metrizable space whose Borel structure is generated by either one of the two families $(U^V)_{V \text{ open }}$ or $(U_K)_{K \text{ compact }}$ (e.g. \cite{bjorklund2016approximate} and references therein). Furthermore, the natural action $(g,X) \mapsto gX$ of $G$ on $\mathcal{C}(G)$ is continuous. 

Given a uniformly discrete subset $X \subset G$, we define the \emph{invariant hull} by 
$$ \Omega_{X}:=\overline{G\cdot X}.$$
The invariant hull is a compact metrizable $G$-space that encodes many properties of $\Omega_X$. For instance, $X$ is relatively dense if and only if $\emptyset \notin \Omega_X$ (see \cite[Prop. 4.4]{bjorklund2016approximate}). We will be particularly interested in the situation where there is a $G$-invariant Borel probability measure $\nu$ on $\Omega_X$ such that $\nu(\{\emptyset\})=0$. We will call such a measure \emph{proper}.

In what follows, we will use the so-called \emph{periodization maps} to draw quantitative information about $X$ from ergodic theorems on $\Omega_X$. Precisely, the periodization map $\mathcal{P}_X$ sends a continuous function with compact support $\phi \in C_c^0(G)$ to the continuous map 
\begin{align*}
\mathcal{P}_X\phi: &\Omega_X \longrightarrow \mathbb{R} \\
 & Y \longmapsto \sum_{y \in Y} \phi(y)
\end{align*}
with support contained in $\Omega_X \setminus \{\emptyset\}$.

The map $\mathcal{P}_X$ is $G$-equivariant and allows to pull-back measures. We mention the following crucial observation: if $\nu$ is proper $G$-invariant Borel probability measure on $\Omega_X$, then $\mathcal{P}_X^* \nu$ is a non-trivial Haar measure on $G$ \cite[Cor. 5.7]{bjorklund2016approximate}. 

Finally, the definition of $\mathcal{P}_X$ extends in a straightforward manner to positive functions $\phi$ on $G$ with $\mathcal{P}_X\phi$ possibly taking infinite values. Note that, when $\phi=\mathbf{1}_V$ for some open relatively compact subset $V$ and $X$ is uniformly discrete, $\mathcal{P}_X\phi$ takes finite values, is upper-semi-continuous and equal to $Y \mapsto |Y \cap V|$.
%
%

\subsection{Counting points with the periodization map}

Our next result is a handy observation that will allow us to translate cardinality estimates into convolution estimates involving quantities reminiscent of the periodization maps: 

\begin{lemma}\label{Lemma: bound convolution and counting}
Let $X$ be a countable subset of a unimodular locally compact group. Let $A,B \subset G$ be two Borel subsets. Then:
\begin{align}
\int_B \sum_{x \in X} \mathbf{1}_{Ax}(g)d\mu_G(g) \leq \mu_G(A)|X \cap A^{-1}B| \label{Eq: lower bound convolution and counting}\\
\mu_G(A)|X \cap B| \leq \int_{AB}\sum_{x \in X} \mathbf{1}_{Ax}(g)d\mu_G(g) \label{Eq: upper bounded convolution and counting}
\end{align}
\end{lemma}

\begin{proof}
To prove (\ref{Eq: lower bound convolution and counting}) notice that 
if $Ax \cap B \neq \emptyset$ then $x \in A^{-1}B$. So 
\begin{align*}
\int_B \sum_{x \in X} \mathbf{1}_{Ax}(g)d\mu_G(g) & = \sum_{x \in X}\int_B \mathbf{1}_{Ax}(g)d\mu_G(g) \\ 
& = \sum_{x \in X \cap A^{-1}B}\int_B \mathbf{1}_{Ax}(g)d\mu_G(g)\\
& \leq \mu_G(A)|X \cap A^{-1}B|.
\end{align*}

Let us now prove (\ref{Eq: upper bounded convolution and counting}). If $x \in B$, then $Ax \in AB$. Therefore, 
\begin{align*}
\mu_G(A)|X \cap B| &\leq \int_{AB}\sum_{x \in X \cap B} \mathbf{1}_{Ax}(g)d\mu_G(g) \\
&\leq \int_{AB}\sum_{x \in X} \mathbf{1}_{Ax}(g)d\mu_G(g).
\end{align*}
\end{proof}

\subsection{Invariant hull and commensurability}

In \cite[\S 2.2.2]{machado2020apphigherrank} we proved a commensurability criterion for uniformly discrete subsets of finite co-volume:

\begin{lemma}\label{Lemma: Hull and commensurability}
Let $X,Y$ be two uniformly discrete subsets of a locally compact group $G$.  Suppose that there is a Borel probability measure $\nu$ on $\Omega_X$ such that:
\begin{enumerate}[label=\alph*)]
\item $\mathcal{P}_X^*\nu \geq \mu_G$ for a Haar measure $\mu_G$ on $G$; 
\item $V \subset G$ is an open subset such that there is $C > 0$ with $|XY \cap gV| \leq C$ for all $g\geq 0$. 
\end{enumerate}
Then $Y$ is covered by $\mu_G(V)^{-1}C$ right-translates of $X^{-1}X$. 
\end{lemma}

Although our hypotheses are slightly weaker, the proof of Lemma \ref{Lemma: Hull and commensurability} is identical to the proof of \cite[Prop. 4]{machado2020apphigherrank}. This is in fact also a special case of a more general result at the heart of the proof of superrigidity and property (T) for $\star$-approximate lattices that we will invoke once more later on, see Section \ref{Section: The invariant hull, equivariant families and cross-sections}.

\subsection{Commensurability and finite co-volume}

We will prove a technical result about commensurable subsets and finite co-volume. This will turn out useful in the proof of Theorem \ref{Theorem: Lagarias-type theorem in amenable groups, hull version}.

\begin{lemma}\label{Lemma: finite co-volume passes to commensurability}
Let $X \subset G$ be uniformly discrete of a locally compact group and suppose that there is a Borel subset $\mathcal{F}$ of finite Haar measure such that $\Omega_X$ admits a proper Borel probability measure $\nu$ such that $\mathcal{P}_X^*\nu \geq \mu_G$ where $\mu_G$ denotes some Haar measure. Let $Y$ be uniformly discrete and let $F$ be a finite subset such that $X \subset FY$. Then there is a Borel subset $\mathcal{F}$ of finite Haar measure such that $Y^{-1}Y\mathcal{F}= G$ and the multiplication map $Y \times \mathcal{F}\rightarrow G$ is one-to-one. 
\end{lemma}

\begin{proof}
By \cite{hrushovski2020beyond}, there is $\mathcal{F}$ Borel such that $Y^{-1}Y\mathcal{F} = G$ and the multiplication map $Y \times \mathcal{F} \rightarrow  G$ is one-to-one. Indeed, let $V_0$ be a compact neighbourhood of the identity such that $VV^{-1} \cap Y^{-1}Y =\{e\}$ and let $(g_n)_{n \geq 0}$ be a sequence of elements of $G$ such that $G=\bigcup_{n \geq 0} Vg_n$. Define inductively $B_n:=Vg_n \setminus \bigcup_{m<n} Y^{-1}Y B_m$ and $B:=\bigcup_{n\geq 0}B_n$. Then $BB^{-1} \cap Y^{-1}Y =\{e\}$ and $Y^{-1}YB=G$. We can in fact choose such an $\mathcal{F}$ with null boundary.  Therefore, the multiplication map $X \times \mathcal{F}\rightarrow G$ is $|F|$-to-one. So for all $g \in G$, $|gX \cap V^{-1}| \leq |F|$ where $V$ denotes the interior of $\mathcal{F}$. We obtain $|Y \cap V^{-1}| \leq |F|$ for all $Y \in \Omega_X$ since $V^{-1}$ is open. Hence, 
$$|F| \geq \int_{\Omega_X}|Y \cap V^{-1}| d\nu(Y) = \mathcal{P}_X^*\nu(V^{-1}).$$
But $\mathcal{P}_X^*\nu$ is a Haar measure. Thus, $V$ has finite Haar measure, and so has $\mathcal{F}$.
\end{proof}

\section{Lagarias-type result in amenable groups}\label{Section: Lagarias-type result in amenable groups}

\subsection{Building measures on the invariant hull}

Our first result already shows that counting points is related to invariant measures on the hull.

\begin{proposition}\label{Proposition: Statistical data and measures on the hull}
Let $X$ be a uniformly discrete subset of a locally compact second countable group $G$. 
Suppose that there is a F\o lner sequence $(F_n)_{n \geq 0}$ such that 
$$\liminf \frac{|X^{-1} \cap F_n|}{\mu_G(F_n)}=\alpha > 0.$$
Then there is a proper $G$-invariant Borel probability measure $\mu$ on $\Omega_X$ such that $\mathcal{P}_X^*\mu \geq \alpha \mu_G$. 
\end{proposition}

\begin{proof}
Take $V$ a compact neighbourhood of the identity such that $X^{-1}X \cap V^{-1}V =\{e\}$. Then we can reformulate (\ref{Eq: upper bounded convolution and counting}) of Lemma \ref{Lemma: bound convolution and counting} as follows: 
\begin{equation}
\mu_G(V) |X^{-1} \cap F_n| \leq \int_{VF_n}\mathcal{P}_X \mathbf{1}_{V}(gX)d\mu_G(g). \label{Eq: Counting folner sets and periodization maps}
\end{equation}
Consider $\phi \in C^0_c(G)$ with non-negative values and $\phi \geq \mathbf{1}_{V}$. Let $\mu_n$ denote the measure with density $\mu_G(VF_n)^{-1}\mathbf{1}_{VF_n}$ against the Haar measure and write $$\liminf \frac{|X \cap F_n|}{\mu_G(F_n)} = \alpha.$$ Since $F_n$ is a F\o lner sequence, $\mu_G(VF_n) \sim \mu_G(F_n)$ and $(VF_n)_{\geq 0}$ is a F\o lner sequence. Then (\ref{Eq: Counting folner sets and periodization maps}) implies that 

\begin{align*}
\liminf\int_{\Omega_X} \mathcal{P}_X\phi(Y)d(\mu_n * \delta_X)(Y) &\geq  \liminf\int_{VF_n} \frac{\mathcal{P}_X \mathbf{1}_{V}(g)}{\mu_G(VF_n)}d\mu_G(g)\\
&\geq  \alpha\mu_G(V) > 0
\end{align*}
where $\delta_X$ denotes the dirac mass at $X$. Therefore, any weak-* limit $\nu$ of $\mu_n * \delta_X$ satisfies 

\begin{equation}
\int_{\Omega_X} \mathcal{P}_X\phi(Y)d\nu(Y) \geq \alpha\mu_G(V). \label{Eq: Inequality on the limit measure}
\end{equation}
 Since $\mathcal{P}_X\phi(Y)$ has compact support contained in $\Omega_X \setminus \{\emptyset\}$, $\nu$ is not supported on $\{\emptyset\}$. But $\nu$ is a $G$-invariant Borel probability measure since $(VF_n)_{n \geq 0}$ is a F\o lner sequence. Note moreover that (\ref{Eq: Inequality on the limit measure}) is valid for any such $\phi$, so 
 $$\mathcal{P}_X^*\nu(V) = \int_{\Omega_X} \mathcal{P}_X\mathbf{1}_V(Y)d\nu(Y) \geq \alpha \mu_G(V).$$
Thus, a proper $G$-invariant ergodic Borel probability measure $\nu'$ such that $\nu'(U^V) \geq \alpha\mu_G(V)$ must appear in the ergodic decomposition of $\nu$. The result now follows from the fact that $\mathcal{P}_X^*\nu'$ is a Haar measure.
\end{proof}

\subsection{The small doubling criterion}
Our proof will rely on the following criterion:

\begin{proposition}\label{Proposition: Technical version Lagarias}
Let $X$ be a subset of a unimodular amenable second countable locally compact group $G$. If:
\begin{enumerate}
\item there is a neighbourhood of the identity $V \subset G$ with $(X^{-1}X)^2 \cap W^{-1}W=\{e\}$; 
\item we have $$\limsup \frac{|X \cap F_n|}{\mu_G(F_n)}=\alpha >0 $$ for some bi-Følner sequence $F_n$;
\end{enumerate}
then there is a subset $S \subset XX^{-1}$ such that $S^n$ is uniformly discrete for all $n \geq 0$ and there is a finite subset $F$ with $X \subset SF$. Moreover, $S^2$ is an $O(\alpha^{-12}\mu_G(W)^{-12})$-approximate subgroup and $F$ has size at most $O(\alpha^{-3}\mu_G(W)^3)$.
\end{proposition}

Our strategy will be to rephrase conditions (1) and (2) in additive-combinatorial terms. However, contrary to the case treated by Konieczny in \cite{konieczny2021characterisation}, we will not be able to write $X$ as an inductive limit of finite approximate subgroups. Rather, we will show that the subsets $(X \cap F_n)_{n \geq 0}$ satisfy a small doubling condition \emph{as a family}. We will build upon an idea of Tao \cite{MR2501249} used in his proof of a non-commutative generalisation of Pl\"{u}nnecke's lemma:

\begin{lemma}\label{Lemma: Plunnecke for family of sets}
Let $A_0, \ldots, A_{k+1}$ be finite subsets of a group $G$. Suppose that there is $K \geq 0$ such that for all $l \in \{0, \ldots, k\}$ we have
$$ |A_l^{-1}A_{l+1}| \leq K|A_{l+1}|.$$ 
For every  $l \in \{1,\ldots,k\}$ set
$$S_l := \{g \in G : |A_l \cap gA_l| \geq (2K)^{-1}|A_l|\}.$$ Then 
$$|A_0^{-1}S_1\cdots S_{k}A_{k+1}| \leq 2^kK^{2k+1}|A_{k+1}|.$$
\end{lemma}

\begin{proof}
Let $m: A_0^{-1}A_1 \times \cdots \times A_{k}^{-1}A_{k+1} \rightarrow G$ denote the multiplication map. Take an element $a \in A_0^{-1}S_1\cdots S_{k}A_{k+1}$ and choose $\alpha_0 \in A_0, \alpha_{k+1} \in A_{k+1}, s_1\in S_1, \ldots, s_k \in S_k$ such that $a=\alpha_0s_1\cdots s_k\alpha_{k+1}$. Now for every $l \in \{1, \ldots, k\}$ take $\alpha_l \in A_l \cap s_lA_l$.  We have $\alpha_l^{-1}s_l\alpha_{l+1} \in A_l^{-1}A_{l+1}$. Moreover, 
$$a = m(\alpha_0^{-1}s_1\alpha_1,\ldots,\alpha_k^{-1}s_k\alpha_{k+1}).$$ So $a$ belongs to the range of the map $m$ and $m^{-1}(a)$ has size at least $|A_1 \cap s_1A_1| \cdots |A_k \cap s_kA_k|$. Therefore, 
$$|m^{-1}(a)| \geq (2K)^{-k}|A_1|\cdots |A_k|.$$ But $$\left|A_0^{-1}A_1 \times \cdots \times A_{k}^{-1}A_{k+1}\right| \leq K^{k+1}|A_1| \cdots |A_{k+1}|.$$
So there are at most 
$2^kK^{2k+1}|A_{k+1}|$ such elements $a$.
\end{proof}

\begin{proof}[Proof of Proposition \ref{Proposition: Technical version Lagarias}.]
We will rephrase (2) as follows. Choose $\alpha > 0$ such that $\limsup |X \cap F_n|/\mu_G(F_n) > \alpha$. Upon considering a subsequence of $(F_n)_{n \geq 0}$ we may assume that 
$$|X \cap F_nW| > \alpha \mu_G(F_nW)$$
for all $n \geq 0$.

We will be able to rephrase (1) in a similar fashion. Fix $\beta_0 > 1$. First of all, upon considering a further subsequence, we may assume that the sequence $(F_n)_{n \geq 0}$ satisfies the Shulman assumption. In other words, 
$$ \forall m > n \geq 0, \mu_G(F_n^{-1}F_m) \leq \beta_0\mu_G(F_m).$$
Note that $\beta_0 > 0$ may be chosen as close to $1$ as one wishes. 

Let $W$ be as in (1). Since $(F_n)_{n \geq 0}$ is bi-F\o lner, we may furthermore assume that
$$ \forall m > n \geq 0,\ \mu_G(W^{-1}F_n^{-1}F_mW^2) \leq \beta_0\mu_G(F_m)$$
for the same $\beta_0$ as above. 

 We have 
\begin{align*}
 \left|X^{-1}X \cap W^{-1}F_n^{-1}F_mW\right|& = \mu_G\left( \left(X^{-1}X \cap W^{-1}F_n^{-1}F_mW\right)W \right) \\
 & \leq \mu_G(W^{-1}F_n^{-1}F_mW^2) \\
 & \leq \beta_0\mu_G(F_m). 
 \end{align*}

Define now the subsets $\tilde{F}_n:=F_nW$ and $A_n:= X \cap \tilde{F}_n$ for all $n \geq 0$. The above discussion implies: 
\begin{align}
&\forall n \geq 0,\ \  \  \  \  \  \  \  |A_n| \geq \alpha \mu_G(F_n) \label{Eq: Lower bound folner sets} \\
&\forall m > n \geq 0, \  |A_n^{-1}A_m| \leq \beta \mu_G(F_m)\label{Eq: Upper bound folner sets}
\end{align}
where $\beta = \frac{\beta_0}{\mu_G(W)}$.

Choose a non principal ultrafilter $U$ on $\mathbb{N}$. Let $P_{W}(G)$ be the set of subsets $Y$ of $G$ such that the multiplication map $Y \times W \rightarrow G$ is one-to-one. For $Y \in P_{W}(G)$ define
$$M(Y) = \lim_{n \rightarrow U} \frac{\mu_G(YW \cap \tilde{F}_nW)}{\mu_G(\tilde{F}_n)}.$$
Since $(F_n)_{n \geq 0}$ is F\o lner, it easily seen that $M$ is a left-invariant finitely additive measure defined on $P_{W}(G)$. According to (\ref{Eq: Lower bound folner sets}) we have $M(X) \geq \alpha\mu_G(W)$. By (\ref{Eq: Upper bound folner sets}) we have $M(X^{-1}X)\leq \beta$. Define
$$ S:=\left\{ g \in G : M(gX \cap X) \geq \frac{2\alpha}{3\beta}M(X)\right\}. $$
By an argument due to Massicot--Wagner (\cite[Proof of Thm 12]{MR3345797}) there is a finite subset $F$ of size $O(\beta^3\alpha^{-3})$ such that $X^{-1}\subset FS$. 

We will now show that $S^n$ is uniformly discrete for all $n \geq 0$. To do so, we will prove that $S^n \cap W$ is finite for all $n \geq 0$. Notice first that $S\subset \bigcup_{n \geq 0} S_n$ where 
$$S_n:=\{g \in G : |A_n \cap gA_n| \geq \frac{\alpha}{2\beta}|A_n|\}$$
for all $n \geq 0$. More precisely, each element of $S$ is contained in infinitely many $S_n$'s. Indeed, for every $Y \in P_W(G)$ we have 
 $$ |Y \cap \tilde{F}_n| \leq\mu_G(YW \cap F_nW^2).$$
 And 
 \begin{align*}
 \mu_G(YW \cap F_nW^2) & \leq |Y \cap \tilde{F}_n| + \mu_G(YW \cap (F_nW^2 \setminus F_n)) \\
 & = |Y \cap \tilde{F}_n| + o(\mu_G(F_n)).
 \end{align*}
 
 But, if $g \in S$, for some $A \in U$ and all $n \in A$ we have 
 $$\mu_G\left( (X \cap gX)W \cap \tilde{F}_nW \right) \geq \frac{3\alpha}{5\beta} \mu_G\left( XW \cap \tilde{F}_nW \right)$$
 which yields
 $$ |X \cap gX \cap \tilde{F}_n| + o(\mu_G(F_n)) \geq \frac{3\alpha}{5\beta} |X \cap \tilde{F}_n|.$$
 Since $|X \cap \tilde{F}_n| \geq \alpha \mu_G(\tilde{F}_n)$, if $n$ is chosen sufficiently large with respect to $U$, then 
 $$ |X \cap gX \cap \tilde{F}_n| \geq \frac{\alpha}{2\beta} |X \cap \tilde{F}_n|.$$
 So our claim is proved up to considering a subsequence.

  Choose any finite subset $F' \subset S^n \cap W$ such that the subsets $fX$ are pairwise disjoint when $f$ runs through $F'$. Then there are $k_n >\ldots > k_0 \geq 0$ such that $F' \subset S_{k_1}\cdots S_{k_n}$. Choose an integer $k_{n+1}$ such that $k_n < k_{n+1}$. We thus have by Lemma \ref{Lemma: Plunnecke for family of sets} (note that we had to specify $k_0$ to invoke Lemma \ref{Lemma: Plunnecke for family of sets}, although it is not needed from now on):
$$|F'||A_{k_{n+1}}| = |F'A_{k_{n+1}}| \leq |S_{k_1}\cdots S_{k_{n}}A_{k_{n+1}}| \leq 2^n\frac{\beta^{2n+1}}{\alpha^{2n+1}}|A_{k_{n+1}}|.$$
Therefore, $|F'|\leq 2^n\frac{\beta^{2n+1}}{\alpha^{2n+1}}$. So take $F'$ as above of maximal size. By Ruzsa's covering lemma we have $S^n\cap W \subset F'XX^{-1}$. Since $F' \subset W$ we finally have 
$$S^n\cap W \subset F'\left(XX^{-1} \cap W^2\right)=F'.$$

Let us finally prove that $S^2$ is an approximate lattice. Notice first that 
\begin{align*}
\limsup \frac{|S \cap F_nF|}{\mu_G(F_n)} & \geq \limsup \frac{|SF^{-1} \cap F_n|}{|F|\mu_G(F_n)} \\
                                              & \geq \limsup \frac{|X \cap F_n|}{|F|\mu_G(F_n)} \\
                                              & \geq \limsup \frac{|X \cap F_n|}{|F|\mu_G(F_n)} > 0.
\end{align*}
Since $(F_n)_{n \geq 0}$ is bi-F\o lner, $(F_nF)_{ n \geq 0}$ is a F\o lner sequence and $\mu_G(F_nF) \sim \mu_G(F_n)$. Hence, according to Proposition \ref{Proposition: Statistical data and measures on the hull} there is a proper $G$-invariant Borel probability measure $\nu$ on $\Omega_S$. By Lemma \ref{Lemma: Hull and commensurability}, that $S^5$ is uniformly discrete implies that $S^2$ is an approximate subgroup and, hence, an approximate lattice. The bounds in the statement of Proposition \ref{Proposition: Technical version Lagarias} are a consequence of the bounds on the size of $F'$ and the ones found in Lemma \ref{Lemma: Hull and commensurability}.
\end{proof}

\begin{remark}\label{Remark: Boxes and nilpotency}
In Konieczny's \cite{konieczny2021characterisation}, a more direct approach when $X \subset \mathbb{R}^n$ relies on the fact that the intersections $X \cap [-R;R]^n$ for $R \rightarrow \infty$ are $l$-approximate subgroups for some fixed $l \geq 0$ independent of $R$. Unfortunately, this fact is specific to nilpotent ambient groups $G$, as can be derived from the polynomial growth theorem for approximate subgroups (see \cite[Appendix]{cordes2020foundations}). 
\end{remark}

\subsection{Double counting and ergodic theorems}

\begin{proposition}\label{Proposition: Lagarias with fundamental domain}
Let $X$ be a subset of a unimodular amenable locally compact second countable group $G$. Suppose that $XX^{-1}$ is uniformly discrete and that there is $\mathcal{F}$ of finite Haar measure such that $X\mathcal{F}=G$. Then for all F\o lner sequences $(F_n)_{n \geq 0}$, 
$$ \liminf \frac{|X \cap F_n |}{\mu_G(F_n)} \geq \frac{1}{\mu_G(\mathcal{F})}.$$
\end{proposition}

Proposition \ref{Proposition: Lagarias theorem with measures on the Hull} will be an easy consequence of a Fubini-type argument that we write down below. 

\begin{lemma}\label{Lemma: Double counting and finite co-volume}
Let $X$ be a subset such that $\mathcal{F}X= G$ for $\mathcal{F}$ of finite co-volume and $X^{-1}X\cap V^{-1}V =\{e\}$ for a neighbourhood $V$ of the identity . Choose $\epsilon >0$. Then there is $K \subset \mathcal{F}$ compact such that for all $B \subset G$ Borel we have 
$$\left| kX \cap BV \right| \geq \frac{(1-\epsilon)\mu_G(B)}{\mu_G(K)}$$
for some $k \in K$.
\end{lemma}

\begin{proof}
Since $\mu_G$ is bi-invariant, for all $h \in G$,
\begin{align*}\int_{\mathcal{F}}\sum_{x \in X} \mathbf{1}_V(hgx) d\mu_G(g) & = \sum_{x \in X} \int_{\mathcal{F}}\mathbf{1}_V(hgx)d\mu_G(g) \\
& = \sum_{x \in X} \int_{\mathcal{F}x}\mathbf{1}_V(hg)d\mu_G(g) \\
& \geq \int_{\mathcal{F}X}\mathbf{1}_V(hg)d\mu_G(g) = \mu_G(V).
\end{align*}
If now $K \subset \mathcal{F}$, we have 
\begin{align*}
\int_{K}\sum_{x \in X} \mathbf{1}_V(hgx) d\mu_G(g) & \geq \mu_G(V) - \int_{\mathcal{F}\setminus K}\sum_{x \in X} \mathbf{1}_V(hgx) d\mu_G(g).
\end{align*}
But
$$\sum_{x \in X}\mathbf{1}_V(hgx) = \left| X \cap g^{-1}h^{-1}V \right|$$
And, if $\mathbf{1}_V(hgx) > 0$, then $ x \in g^{-1}h^{-1}V$. Hence,
 $$\sum_{x \in X}\mathbf{1}_V(hgx) \leq \left|X^{-1}X\cap V^{-1}V\right| = 1.$$
So 
\begin{align*}
\int_{K}\sum_{x \in X} \mathbf{1}_V(hgx) d\mu_G(g) \geq \mu_G(V) - \mu_G(\mathcal{F}\setminus K).
\end{align*}
Consider $K$ compact such that $\mu_G(\mathcal{F}\setminus K) \leq \epsilon \mu_G(V)$. Then 

$$\int_{B^{-1}} \int_{K}\sum_{x \in X} \mathbf{1}_V(hgx) d\mu_G(g) d\mu_G(h)\geq (1-\epsilon)\mu_G(B)\mu_G(V).$$
On the other hand, Lemma \ref{Lemma: bound convolution and counting} implies,
$$ \int_{B^{-1}}\sum_{x \in X}\mathbf{1}_{V}(hgx) d\mu_G(h) \leq |gX \cap BV|\mu_G(V).$$
Therefore, 
$$\int_K \mu_G(V) \left|gX \cap BV \right| d\mu_G(g) \geq (1-\epsilon)\mu_G(B)\mu_G(V)$$
which concludes the proof.
\end{proof}

\begin{proof}[Proof of Proposition \ref{Proposition: Lagarias with fundamental domain}.]
Let $(K_n)_{n \geq 0}$ be an exhaustion of $G$ by compact subsets and let $V$ be a compact neighbourhood of the identity such that $XX^{-1} \cap V^{-1}V =\{e\}$. Fix $n \geq 0$. The set $X$ satisfies the hypotheses of Proposition \ref{Lemma: Double counting and finite co-volume}. Let $K_n'$ be given by Proposition \ref{Lemma: Double counting and finite co-volume} applied to $X$ and $\epsilon=1/n$. 

Since $G$ is amenable, there is a symmetric compact subset $F_n$ such that $$\mu_G(K_nK_n'F_nVK_n) \leq (1+\frac{1}{n})\mu_G(F_n).$$

Define now $F_n':=K_n'F_nV$ for all $n \geq 0$. Then $(F_n')_{\geq 0}$ is a bi-F\o lner sequence and 
\begin{align*}
|X \cap F_n'|  \geq (1 - \frac{1}{n})\frac{\mu_G(F_n)}{\mu_G(\mathcal{F})}
 \geq \frac{(1 - \frac{1}{n})}{1 + \frac{1}{n}}\frac{\mu_G(F_n')}{\mu_G(\mathcal{F})}.
\end{align*}
\end{proof}

\begin{proof}[Proof of Theorem \ref{Theorem: Lagarias-type theorem in amenable groups, first version}.]
By Proposition \ref{Proposition: Lagarias with fundamental domain} for all F\o lner sequences $(F_n)_{n \geq 0}$, 
$$ \liminf \frac{|X \cap F_n |}{\mu_G(F_n)} \geq \frac{1}{\mu_G(\mathcal{F})}.$$
So Theorem \ref{Theorem: Lagarias-type theorem in amenable groups, first version} is a consequence of Proposition \ref{Proposition: Technical version Lagarias}. 
\end{proof}

As another consequence  of Proposition \ref{Proposition: Lagarias with fundamental domain} the invariant hull $\Omega_X$ admits a proper $G$-invariant Borel probability measure (see Proposition \ref{Proposition: Statistical data and measures on the hull}). Conversely, when we merely have information on the hull rather than on a ``fundamental domain" $\mathcal{F}$, we can invoke mean ergodic theorems to conclude: 

\begin{proposition}\label{Proposition: Lagarias theorem with measures on the Hull}
Let $X$ be a uniformly discrete subset of $G$ such that $\Omega_X$ admits a proper ergodic $G$-invariant Borel probability measure $\nu$. Let $V$ be a neighbourhood of the identity such that $X^{-1}X \cap V^{-1}V = \{e\}$. Let $\mu_G$ be the Haar measure given by $\mathcal{P}_X^*\nu$. Then there is a bi-F\o lner sequence $(F_n)_{n \geq 0}$ such that 
$$\liminf \frac{|Y \cap VF_n|}{\mu_G(VF_n)} \geq 1$$
for $\nu$-almost every $Y \in \Omega_X$. 
\end{proposition}

\begin{proof} Let $(F_n)_{n \geq 0}$ be a F\o lner sequence and let $\mu_n$ be the probability measure with density $\mu_G(F_n)^{-1}\mathbf{1}_{F_n}$ against the Haar measure. By the mean ergodic theorem, we know that $\phi_n: Y \rightarrow \int_G |gY \cap V|d\mu_n(g)$ converges in the $L^2$-norm to its average $\int_G \mathcal{P}_X\mathbf{1}_V(Y)d\nu(Y)$. But $\mathcal{P}_X\mathbf{1}_V = \mathbf{1}_{U^V}$. So $\phi_n$ converges in the $L^2$-norm to $\nu(U^V)$. Upon considering a subsequence, we may assume that the convergence is $\nu$-almost everywhere. Therefore, for $\nu$-almost all $Y$ we have $\int_G |gY \cap V| d\mu_n(g) \rightarrow \nu(U^V)$. According to Lemma \ref{Lemma: bound convolution and counting} we find:
$$\liminf \frac{|Y \cap VF_n|}{\mu_G(VF_n)} \geq \frac{\nu(U^V)}{\mu_G(V)}=1.$$
\end{proof}

It is interesting to compare the two formulae obtained in Proposition \ref{Proposition: Lagarias theorem with measures on the Hull} and Proposition \ref{Lemma: Double counting and finite co-volume}. Indeed, we see that the Haar measure $\mu_G=\mathcal{P}_X^*\nu$ seem to correspond with a Haar measure normalised so that $X$ has co-volume $1$ \emph{in both meanings of co-volume}.

\begin{proof}[Proof of Theorem \ref{Theorem: Lagarias-type theorem in amenable groups, hull version}.]
According to Proposition \ref{Proposition: Lagarias with fundamental domain}, there is a bi-F\o lner sequence $(F_n)_{n \geq 0}$ such that 
$$\liminf \frac{|Y \cap VF_n|}{\mu_G(VF_n)} \geq 1$$
for $\nu$-almost every $Y \in \Omega_X$ where $\mu_G:=\mathcal{P}_X^*\nu$. Fix one such $Y_0 \in \Omega_X$. By Proposition \ref{Proposition: Technical version Lagarias} applied to $Y_0$ there is an approximate subgroup $S$ contained in $Y_0Y_0^{-1}$ and a finite subset $F$ such that $Y_0 \subset SF$. So $S$ is an approximate lattice according to Lemma \ref{Lemma: finite co-volume passes to commensurability}.

\end{proof}

\section{The invariant hull, equivariant families and cross-sections}\label{Section: The invariant hull, equivariant families and cross-sections}

In this section, we establish additional results regarding invariant hull which we us in the proof of Theorem \ref{Theorem: Lagarias-type theorem in simple groups}. In the following, consider $G$ a second countable locally compact group and $X \subset G$ such that $\Omega_X$ admits a proper ergodic $G$-invariant Borel probability measure, say $\nu$. Assume moreover that both $X^{-1}X$ and $XX^{-1}$ are uniformly discrete. We start with a general result concerning equivariant families of discrete subsets. 

\subsection{Equivariant families of discrete subsets}

\begin{lemma}\label{Lemma: Key lemma equivariant maps of subsets}
Let $S$ be a Borel $G$-space with a probability measure preserving action of a unimodular locally compact second countable group $G$. Let $\Phi: S\rightarrow \Omega_X$ be a Borel $G$-equivariant map taking non-empty values for almost every $s$. Let $Y \subset G$ be such that $XY^{-1}$ is uniformly discrete. Then for all $Y' \in \Omega_Y$, $Y'$ is covered by finitely many left-translates of $\bigcup_{s \in S}\Phi(s)^{-1}\Phi(s)$. If, moreover, the action is ergodic, then for almost every $s \in S$, $Y'$ is covered by finitely many left-translates of $\Phi(s)^{-1}\Phi(s)$.
\end{lemma}

The proof of Lemma \ref{Lemma: Key lemma equivariant maps of subsets} is a first instance of the overall strategy. Below we will show how using periodization maps will allow us to reformulate the problem into a simple counting question, which we will then solve using elementary covering arguments. 

The proof method will essentially follow the method found in \cite{machado2020apphigherrank}. There, we proved a similar result valid more generally for maps that are not $G$-equivariant but satisfy a functional identity. 

\begin{proof}
Let $\nu_0$ denote a $G$-invariant probability measure on $S$.  Take $F \subset (Y')^{-1}$ finite such that for $\nu_0$-almost every $s \in S$ the subsets $(\Phi(s)f)_{f \in F}$ are pairwise disjoint. Let $V \subset G$ be a neighbourhood of the identity such that $YX^{-1}XY^{-1} \cap V^{-1}V =\{e\}$. Therefore, for all $g \in G$ and $X' \in \Omega_X$, we have $|X'Y^{-1} \cap gV| \leq 1$ for all $X' \in \Omega_X$. Since $F$ is finite and $G \cdot Y$ is dense in $\Omega_Y$, there is $h \in G$ such that $F \subset Y^{-1}h$. Now we have 
\begin{align*}\int_S \sum_{x \in \Phi(s)F}\mathbf{1}_{Vh}(x)d\nu_0(s)&  =  \int_S \sum_{x \in \Phi(s)Fh^{-1}}\mathbf{1}_{V}(x)d\nu_0(s) \\
& \leq   \int_S \sum_{x \in \Phi(s)Y^{-1}}\mathbf{1}_{V}(x)d\nu_0(s) \\
& \leq 1.
\end{align*}
On the other hand, 
\begin{align*}
\int_S \sum_{x \in \Phi(s)F}\mathbf{1}_{Vh}(x)d\nu_0(s) & = \sum_{f \in F} \int_S \sum_{x \in \Phi(s)f}\mathbf{1}_{Vh}(x)d\nu_0(s) \\
& = \sum_{f \in F} \int_S \sum_{x \in \Phi(s)}\mathbf{1}_{Vhf^{-1}}(x)d\nu_0(s) \\
& = \sum_{f \in F} \mathcal{P}_X^*\Phi_*\nu_0(Vhf^{-1}) = |F|\mathcal{P}_X^*\Phi_*\nu_0(V)
\end{align*}
where the interversion in the first line was because of the disjointness property of $F$ and we use in the last line that $\mathcal{P}_X^*\Phi_*\nu_0$ is a Haar measure of a unimodular locally compact group. Thus, $|F| \leq \mathcal{P}_X^*\Phi_*\nu_0(V)^{-1}$. Hence, $$(Y')^{-1} \subset \bigcup_{s \in S} \Phi^{-1}(s)\Phi(s)F$$ by the covering lemma (Lemma \ref{Lemma: Ruzsa's covering lemma for families of sets}). 

Suppose now that $\nu_0$ is ergodic. For all $g \in G$ the set 
$$S_g:=\{s \in S : g \in \Phi^{-1}(s)\Phi(s)\}$$ is $G$-invariant and Borel. Indeed, since $\Phi(s)^{-1}\Phi(s)$ is a uniformly discrete subset, we have $g \in \Phi^{-1}(s)\Phi(s)$ if and only if 
$$\Phi(s) \in \bigcap_{n \geq 0}\bigcup_{d \in D} U^{dV_n} \cap U^{dV_ng}$$ where $(V_n)_{n \geq 0}$ is a neighbourhood basis at $e$ and $D$ is a dense subset of $G$. Because $\nu$ is ergodic we thus have that $S_g$ is either null or co-null. Now, $\bigcup_{s \in S} \Phi^{-1}(s)\Phi(s)$ is contained in  $X^{-1}X$ which is countable. Define 
$$\tilde{S}=\bigcap_{g : \nu(S_g)=1}S_g \setminus \left( \bigcup_{g : \nu(S_g)=0}S_g\right).$$
Then for every $s \in \tilde{S}$ we have $\bigcup_{s' \in \tilde{S}} \Phi^{-1}(s')\Phi(s') = \Phi^{-1}(s)\Phi(s)$. Hence, $$Y' \subset \bigcup_{s' \in \tilde{S}} \Phi^{-1}(s')\Phi(s')F = \Phi^{-1}(s)\Phi(s)F$$ for any $s \in \tilde{S}$. 
\end{proof}

\subsection{Canonical transverse and hitting times}

 The proof of Theorem \ref{Theorem: Lagarias-type theorem in simple groups} will make central use of the following notion studied by Bj\"{o}rklund, Hartnick and Karasik in \cite{https://doi.org/10.48550/arxiv.2108.09064}: 
 
 \begin{definition}\label{Definition: Canonical transverse and its null subsets}
 The \emph{canonical transverse} is defined as the subset 
 $$\mathcal{T}_e:=\{Y \in \Omega_X : e \in Y\}.$$ A Borel subset $B \subset \mathcal{T}_e$ is said \emph{null} if $\mu(GB)=0$.
 \end{definition} 
 
Let us give another characterisation of $\mathcal{T}_e$. If $Y \in \mathcal{T}_e$, then there is a sequence $(g_n )_{n \geq 0}$ of elements of $G$ such that $g_nX \rightarrow Y$. But $e\in Y$, so we can find a sequence $x_n \in X$ such that $g_nx_n \rightarrow e.$ In turn, this means that $x_n^{-1}X \rightarrow Y$. Hence, 
$$ \mathcal{T}_e = \overline{\{x^{-1}X : x \in X\}}.$$

 In general, $\mathcal{T}_e$ is a compact subset of $\Omega_X$ and satisfies:
 \begin{enumerate}
 \item $G\mathcal{T}_e = \Omega_X$;
 \item $W \times \mathcal{T}_e \rightarrow \Omega_X$ is one-to-one if $W^{-1}W \cap Y_1Y_2^{-1} = e$ for every $Y_1,Y_2 \in \mathcal{T}_e$.
 \end{enumerate}

  But, under the assumptions of Theorem \ref{Theorem: Lagarias-type theorem in simple groups}, there is a neighbourhood $W$ of the identity such that $W^{-1}W \cap x_1^{-1}XX^{-1}x_2 = \{e\}$ for all $x_1, x_2 \in X$. So (2) holds for this choice of $W$. Therefore, for any Borel subset $B \subset \mathcal{T}_e$, $G B$ is a Borel subset of $\Omega_X$ (\cite[Corollary A.6]{zimmer2013ergodic}). In other words, $\mathcal{T}_e$ is a cross-section.
  
  \begin{remark}
   Rather than simply defining null subsets of the canonical transverse, they study in \cite{https://doi.org/10.48550/arxiv.2108.09064} the notion of transverse measure that formalises the idea of restricting a measure on $\Omega_X$ to a measure on $\mathcal{T}_e$. In fact, there is on $\mathcal{T}_e$ a finite Borel measure $\eta$ such that the restriction of $\nu$ to $U^W=W\mathcal{T}_e$ is equal to the product measure $(\mu_G)_{|W} \otimes \eta$ (see e.g. \cite[\S 4]{https://doi.org/10.48550/arxiv.2108.09064}). In particular, if $B \subset \mathcal{T}_e$ is measurable and $g \in G$ satisfies $gB \subset \mathcal{T}_e$, then $\eta(gB)=\eta(B)$. Moreover, a measurable subset $B \subset \mathcal{T}_e$ is null in the sense of Definition \ref{Definition: Canonical transverse and its null subsets} if and only if $\eta(B)=0$.
   \end{remark}
   
   We will only use a simple consequence of the existence of a transverse measure.
   
   \begin{lemma}\label{Lemma: Existence support on the canonical transverse}
   Let $\phi: \mathcal{T}_e \rightarrow H$ be a measurable map taking values in a second countable topological space $H$. Then there is $h \in H$ such that for any neighbourhood $V \subset H$ of $h$ the subset $\phi^{-1}(V)$ is non-null. 
   \end{lemma}
   
   Lemma \ref{Lemma: Existence support on the canonical transverse} essentially reduces to saying that the support of the push-forward of the transverse measure is not empty. But this is guaranteed by the second countability of $H$.

 \begin{definition}
 For $Y \in \Omega_X$ and $B \subset \mathcal{T}_e$ the set of \emph{hitting times} is the subset of $G$ defined as follows,
 $$T_h(Y,B):=\{g \in G : gY \in B\}.$$
 \end{definition}
 
 We list now a number of properties satisfied by $T_h$. 
 
 \begin{lemma}\label{Lemma: Miscellaneous hitting times sets}
 With $X,\mathcal{T}_e$ as above. Take $Y \in \Omega_X$ and $B \subset \mathcal{T}_e$. Then:
 \begin{enumerate}
 \item for all $g \in G$, 
 $$T_h(gY,B) = T_h(Y,B)g^{-1};$$
 \item if $g \in T_h(Y,B)T_h(Y,B)^{-1}$ then there is $Y' \in B$ such that $gY'\in B$ i.e. 
 $$g \in \mathcal{R}(B)$$;
 \item we have the inclusions,
 $$T_h(Y,B)T_h(Y,B)^{-1} \subset \bigcup_{Z \in B} Z^{-1} \subset X^{-1}X;$$
 
 \item and,  for $W \subset G$ open,
  $$\{Y \in \Omega_X : T(Y,B) \cap W \neq \emptyset \} = W^{-1}B.$$ 
 \end{enumerate}
 \end{lemma}
 
 \begin{proof}
 We have the equivalences: 
 $$g_0 \in T_h(gY,B) \Leftrightarrow g_0gY \in B \Leftrightarrow g_0g \in T_h(Y,B).$$ So $T_h(gY,B) = T_h(Y,B)g^{-1}$. This proves (1). Take $g_1, g_2 \in T_h(Y,B)$. Set $Y_1:=g_1Y$ and $Y_2:=g_2Y$. Then $Y_1,Y_2 \in B$ and $g_1g_2^{-1}Y_2 = Y_1$. So taking $Y'=Y_2$ works and (2) is proved. Furthermore, since $e \in Y_1$ there is $y \in Y_2$ such that $g_1g_2^{-1}y=e$ i.e. $g_1g_2^{-1} \in Y_2^{-1}$. But $Y_2 \in B$ so the first inclusion of (3) is proved. If, now, we take any $Z \in B$, then $Z \in \mathcal{T}_e$. In other words, $e \in Z$. By \cite{bjorklund2016approximate}, we have $Z \subset X^{-1}X$. This proves (3). Finally, remark that: 
  $$\{Y \in \Omega_X : T(Y,B) \cap W \neq \emptyset \}  = \{Y \in \Omega_X: WY \cap B \neq \emptyset\} = W^{-1}B.$$
  And (4) is proved.
 \end{proof}

We are now in position to apply Lemma \ref{Lemma: Key lemma equivariant maps of subsets} to obtain information on sets of hitting times. 

\begin{lemma}\label{Lemma: Key lemma hitting times}
Take $B \subset \mathcal{T}_e$ that is not a null subset. Suppose that $Z_0 \subset G$ is any subset such that $XZ_0^{-1}$ is uniformly discrete. Then for any $Z \in \Omega_{Z_0}$ with $e \in Z$ there is $F \subset G$ finite such that $$Z \subset F T_h(Y,B)T_h(Y,B)^{-1}$$
for almost all $Y \in \Omega_X$.
\end{lemma}

\begin{proof}
By Lemma \ref{Lemma: Key lemma equivariant maps of subsets} it suffices to show that: 
\begin{align*}
\Omega_X & \longrightarrow \mathcal{C}(G) \\
Y & \longmapsto T_h(Y,B)^{-1}
\end{align*}
is a well-defined $G$-equivariant Borel map, takes values in $\Omega_{X'}$ for $X'$ such that $X'Z_0$ is uniformly discrete and takes non-empty values almost always. We will see this as a consequence of Lemma \ref{Lemma: Miscellaneous hitting times sets}. Part (1) of Lemma \ref{Lemma: Miscellaneous hitting times sets} implies $G$-equivariance and part (2) implies measurability. Since $\nu$ is ergodic, there is $Y'$ such that for almost all $Y \in \Omega_X$, $T_h(Y,B)^{-1} \in \Omega_{X'}$ with $X'=T_h(Y',B)$. In particular,  $X'^{-1}X' \subset X^{-1}X$  so $X'Z_0^{-1}$ is uniformly discrete. It remains to prove that $T_h(Y,B)$ takes non-empty values for $\nu$-almost all $Y$. But $T_h(Y,B)$ is non-empty as soon as $Y \in G\cdot B$. Since $B$ is not be null, $G\cdot B$ is a $G$-invariant Borel subset of $\Omega_X$ of full measure and, hence, $\nu(G \cdot B)=1$ as $\nu$ is ergodic. 
\end{proof}

\subsection{Cocycles on the invariant hull}
The kind of cocycles we are going to work with were first considered in \cite{MR4245602} to study certain Kazhdan-type and Haager-up-type properties of uniform model sets. In \cite{machado2020apphigherrank} we considered related cocycles on the extended invariant hull to show that superrigidity theorems hold for $\star$-approximate lattices. These cocycles are built from Borel sections of the invariant hull:

\begin{definition}\label{Definition: Section of set of closed subsets}
 A map $s: \Omega_X \rightarrow G$ is a \emph{Borel section} if it is Borel, defined $\nu$-almost everywhere and for $\nu$-almost all $Y \in \Omega_X$ we have $s(Y) \in Y$. Given a Borel section $s$, one can build a cocycle defined for all $g \in G$ and $\nu$-almost all $Y \in \Omega_X$ by
 $$\alpha_s(g,Y) := s(gY)^{-1}gs(Y).$$ 
\end{definition}
As it is stated here, our definition implies that if $\nu$ is not proper, then no Borel section exists. We refer to \cite{MR4245602, machado2020apphigherrank} for this and more concerning these specific cocycles on invariant hulls. 

\subsection{Restricting maps to the canonical transverse}

When $X$ is a subgroup and, thus, $\Omega_X\setminus\{\emptyset\}$ is a transitive space, Zimmer uses a result of Mackey to show that cocycle superrigidity implies Margulis' superrigidity theorem, see \cite{zimmer2013ergodic, MR1090825}. Mackey's result asserts that a cocycle on a transitive space can be completely understood simply by looking at its restriction at one point. Precisely, given $Y \in \Omega_X\setminus\{\emptyset\}$ the map that sends a cocycle $\alpha$ to $\alpha_{|\{x\} \times Stab(x)}$ provides a bijection between cohomology classes of cocycles $ \Omega_X\setminus\{\emptyset\} \times G \rightarrow H$ and group homomorphisms $Stab(x)\rightarrow H$ (\cite[4.2.15, 4.2.16 and 5.2.5]{zimmer2013ergodic}).

When $X$ is not assumed to be a subgroup, $\Omega_X\setminus\{\emptyset\}$ is not a transitive space any more. However, as observed by Bj\"{o}rklund, Hartnick and Karasik (private communication), we can partially generalise Mackey's result by considering the restriction of cocycles of certain types to the canonical transverse.

\begin{lemma}[Bj\"{o}rklund--Hartnick--Karasik]\label{Lemma: Consequence superrigidity}
Let $s: \Omega_X \rightarrow G$ denote a Borel section. Let $\tau: \langle X \rangle \rightarrow \mathbb{H}(k)$ be a group homomorphism where $k$ is a local field and $\mathbb{H}$ is a simple algebraic group defined over $k$. Suppose that the  cocycle $\tau \circ \alpha_s$ satisfies the conclusions of Zimmer's cocycle superrigidity theorem. Namely: 
\begin{enumerate}[label=(\roman*)]
\item either $\tau \circ \alpha_s$ is cohomologous with a cocycle taking values in a compact subgroup of $\mathbb{H}(k)$;
\item or there are a continuous group homomorphism $\pi: G \rightarrow \mathbb{H}(k)$ and a measurable map $\phi: \Omega_X \rightarrow \mathbb{H}(k)$ such that for all $g \in G$ and for $\nu$-almost all $Y \in \Omega_X$, 
$$ \tau \circ \alpha_s(g,Y) = \phi(gY)^{-1}\pi(g)\phi(Y).$$
\end{enumerate}

 Then there are a $G$-invariant Borel subset $\Omega$ of full measure and a measurable map $\psi: \mathcal{T}_e \cap \Omega \rightarrow H$ such that: 
 
\begin{enumerate}
\item in case (i), $\psi(gY)\tau(g)\psi(Y)^{-1}$ takes values in a compact subgroup of $\mathbb{H}(k)$ when $Y$ ranges through $\Omega \cap \mathcal{T}_e$ and $g$ ranges through $T_h(Y, \Omega \cap \mathcal{T}_e)$;
\item in case (ii), there is a continuous group homomorphism $\pi: G \rightarrow \mathbb{H}(k)$ such that for all $ Y \in \Omega \cap \mathcal{T}_e$ and $g \in T_h(Y, \Omega \cap \mathcal{T}_e)$, 
$$ \tau (g) = \psi(gY)^{-1}\pi(g)\psi(Y).$$
\end{enumerate}
\end{lemma}

\begin{proof}
We explain the proof under assumption (ii), the proof of (1) will then follow with minor changes. By a Fubini argument, for $\nu$-almost all $Y \in \Omega_X$ and almost all $g \in G$ we have, 
\begin{equation}
 \tau \circ \alpha_s(g,Y) = \phi(gY)^{-1}\pi(g)\phi(Y). \label{Eq: Equivalence cocycles}
 \end{equation}
According to \cite[Lemma B.8]{zimmer2013ergodic}, we can therefore find a $G$-invariant Borel subset $\Omega \subset \Omega_X$ of measure $1$ and a Borel map $s':\Omega \rightarrow G$ such that for all $Y \in \Omega$ and almost all $g \in G$, (\ref{Eq: Equivalence cocycles}) holds at $(g,s'(Y)Y)$ - in other words, 
\begin{equation}
\tau \circ \alpha_s(g,s'(Y)Y) = \phi(gs'(Y)Y)^{-1}\pi(g)\phi(s'(Y)Y). \label{Eq: Equivalence cocycles 2}
\end{equation} 
Note moreover that $s'(Y)=e$ for almost all $e \in \Omega$. By a further Fubini argument, the $G$-invariant subset $\{Y \in \Omega : s'(gY)=e \text{ for almost all } g \in G\}$ is Borel and co-null in $\Omega$. So we may assume that it is equal to $\Omega$. For all $Y \in \Omega$, set $s''(Y)=s'(Y)^{-1}s(s'(Y)Y)$. Since $s$ is a Borel section we get, 
$$s''(Y) =s'(Y)^{-1}s(s'(Y)Y) \in s'(Y)^{-1}s'(Y)Y = Y.$$
So $s''$ is a Borel section as well. Fix $Y \in \Omega$, for almost all $g \in G$ we have $s'(gY)=e$. Hence,  (\ref{Eq: Equivalence cocycles 2}) holds at $(s'(gY)gs'(Y)^{-1},s'(Y)Y))$ for almost all $g\in G$. So, for all $Y \in\Omega$ and almost all $g \in G$,  (\ref{Eq: Equivalence cocycles 2}) becomes
\begin{equation}
\tau\circ \alpha_{s''}(g,Y) =\tau\circ \alpha_{s}(s'(gY)gs'(Y)^{-1},Y) =  \psi(gY)^{-1}\pi(g)\psi(Y) \label{Eq: Equivalence cocycles 3}
\end{equation}
where $\psi(Y) := \pi(s'(Y))^{-1}\phi(s'(Y)Y)$. Therefore, for all $Y \in \Omega$ and $h \in G$ we have
\begin{align*}
\pi(h)\psi(Y) & = \pi(g)^{-1}\psi(ghY)\tau\circ \alpha_{s''}(gh,Y) &\text{for almost all } g \in G \\
              & = \pi(g)^{-1}\psi(ghY)\tau\circ \alpha_{s''}(g,hY)\tau\circ \alpha_{s''}(h,Y) &\text{for almost all } g \in G \\
              & = \psi(hY)\tau\circ \alpha_{s''}(h,Y)
\end{align*}
 where we have used (\ref{Eq: Equivalence cocycles 3}) in the first line, the cocycle identity to deduce the second line and (\ref{Eq: Equivalence cocycles 3}) again to get to the last line. Notice now that for all $Y \in \Omega$, we have $s(s'(Y)Y)^{-1}s'(Y)s(Y) \in Y^{-1}Y \subset X^{-1}X$. So we can unfold the definitions of $s''$ and $\alpha_{s''}$ to find that, for all $g \in G$ and all $Y \in \Omega$, we have
 $$\tau\circ \alpha_{s}(g,Y) = \chi(gY)^{-1}\pi(g)\chi(Y)$$
 where 
 $$\chi(Y):=\pi(s'(Y))^{-1}\phi(Y)\tau\left(s(s'(Y)Y)^{-1}s'(Y)s(Y)\right).$$ Finally, since for $Y \in \Omega \cap \mathcal{T}_e$ and $g \in T_{h}(Y,\Omega \cap \mathcal{T}_e)$ we have $s(Y)=s(gY)=e$, we get Lemma \ref{Lemma: Consequence superrigidity}.
 \end{proof}

\begin{remark}
One might note that the Borel section chosen at the beginning of Lemma \ref{Lemma: Consequence superrigidity} does not appear in the latter part of the statement. This highlights the easily established fact that if $s'$ is some other Borel section, then $\tau \circ \alpha_s$ and $\tau \circ \alpha_{s'}$ are cohomologous. In addition, the map $\phi$ realising the cohomology relation takes values in $\tau(X^{-1}X)$. 
\end{remark}

We will follow a similar strategy to prove a related result about restrictions of cocycles acting on Banach spaces. We will then use this in our considerations on property (T).

\begin{lemma}\label{Lemma: restriction alpha-invariant vector to transverse}
 Suppose that $\langle X \rangle$ acts by isometries on a Banach space $(B,\ ||\cdot||)$. Let $\phi: \Omega_{X} \rightarrow B$ be a measurable map such that for all $g \in G$ and $\nu$-almost all $Y \in \Omega_X$ we have 
 $$ \alpha_s(g,Y)\phi(Y) = \phi(gY).$$
 Then there are a Borel $G$-invariant subset $\Omega \subset  \Omega_X$ of full measure and $\psi: \Omega \cap \mathcal{T}_e \rightarrow B$ a measurable map such that for all $Y \in \Omega \cap \mathcal{T}_e$ and all $g \in T_h(Y, \Omega \cap \mathcal{T}_e)$ we have 
 $$ g\cdot \psi(Y) = \psi(gY).$$
 Moreover, if $\phi$ takes values in the unit ball of $B$, then so does $\psi$.
 \end{lemma}

 \begin{proof} 
 Let $\mu_G$ denote the Haar measure on $G$. By a Fubini-type argument there is $\Omega_0 \subset \Omega_{X}$ of full measure such that for all $Y \in \Omega_0$ and almost all $ g \in G$ we have $ \alpha_s(g,Y)\phi(Y) = \phi(gY)$.  By \cite[Lemma B.8]{zimmer2013ergodic} we can find $\Omega \subset \Omega_X$ $G$-invariant and of measure $1$ and a Borel map $s':\Omega \rightarrow G$ such that for all $Y \in \Omega$ and almost all $g \in G$ we have 
 $$ \alpha_s(g,s'(Y)Y)\phi(s'(Y)Y) = \phi(gs'(Y)Y).$$ 
 Set $s''(Y):=s'(Y)^{-1}s(s'(Y)Y)$ for all $Y \in \Omega$.  Then $s''$ is a Borel section of $\Omega$ (Definition \ref{Definition: Section of set of closed subsets}) that satisfies $s''(Y)=e$ for all $Y \in \Omega \cap  \mathcal{T}_e$.
 Furthermore, one can ensure - proceeding as in the proof of Lemma \ref{Lemma: Consequence superrigidity} - that for all $Y \in \Omega$ and almost all $g \in G$ we have
$$\alpha_{s''}(g,Y) \phi(s'(Y)Y)= \phi(s'(gY)gY).$$
 Let $\psi: \Omega \rightarrow B$ be defined by $\psi(Y)= \phi(s'(Y)Y)$. For every $Y \in \Omega$, we have  
 $$\alpha_{s''}(g,Y)^{-1}\psi(gY)=\psi(Y)$$ for almost all $g \in G$ . Moreover, we have for all $Y \in \Omega$ and $h \in G$: 
   \begin{align*}\psi(hY) & = \alpha_{s''}(g,hY)^{-1}\psi(ghY) & \text{ for almost all } g \in G\\
   & = \alpha_{s''}(h,Y)\alpha_{s''}(gh,Y)^{-1}\psi(ghY) & \text{ for almost all } g \in G \\
   & = \alpha_{s''}(h,Y)\psi(Y)  \\
   \end{align*}
   where we went from the first line to the second using the cocycle identity. We conclude as in the proof of Lemma \ref{Lemma: Consequence superrigidity}.
 \end{proof}

 \section{Lagarias-type result in simple Lie groups}\label{Section: Lagarias-type theorem in simple Lie groups}
 We will now prove Theorem \ref{Theorem: Lagarias-type theorem in simple groups}. In the following we will consider a subset $X$ of the group of points $G$ of an absolutely simple algebraic group defined over a local field $k$ of characteristic $0$. Assume also that $\Omega_X$ admits a proper ergodic $G$-invariant Borel probability measure, say $\nu$. We will assume moreover that both $X^{-1}X$ and $XX^{-1}$ are uniformly discrete. We start with a general result concerning equivariant families of discrete subsets. 
\subsection{Superrigidity}

We will now prove the superrigidity theorem needed in the last part of the proof. 

\begin{theorem}\label{Theorem: Superrigidity for return times, version 2}
 Let $X \subset G$ be such that $X^{-1}X$, $XX^{-1}$ are uniformly discrete and $\Omega_X$ admits a proper $G$-invariant Borel probability measure $\nu$. Then there is $S \subset X^{-1}X$  such that for every $\tau: \langle X \rangle \rightarrow \mathbb{H}(k)$ group homomorphism with $k$ a local field and $\mathbb{H}$ a simple algebraic group defined over $k$ such that $\dim_k(\mathbb{H}) \leq \dim(G)$ we have: 
 \begin{enumerate}
 \item either $\tau(X)$ is relatively compact;
 \item or there exists $\pi: G \rightarrow \mathbb{H}(k)$ a continuous group homomorphism such that $X \subset FS $ for some finite set $F$ and $\tau_{|S} =\pi_{|S}$.
 \end{enumerate}
 \end{theorem}
 
 \begin{proof}
 Fix a Borel section $s: \Omega_X \rightarrow G$. Since $\dim_k(\mathbb{H}) \leq \dim(G)$ and $G$ is a property (T) simple algebraic group, we can apply \cite{machado2020apphigherrank}. So $\tau \circ \alpha_s$ satisfies the conclusions of Zimmer's cocycle superrigidity. Therefore, we may apply Lemma \ref{Lemma: Consequence superrigidity}.

 If we have a measurable map $\psi: \Omega \cap \mathcal{T}_e \rightarrow \mathbb{H}(k)$ and a compact subgroup $K\subset \mathbb{H}(k)$ such that $\psi(gY)\tau(g)\psi(Y)^{-1} \subset K$ for all $Y \in \Omega \cap \mathcal{T}_e$ and $g \in T_h(Y, \Omega \cap \mathcal{T}_e)$  (part (1) of Lemma \ref{Lemma: Consequence superrigidity}) we will prove that (1) of Theorem \ref{Theorem: Superrigidity for return times, version 2} holds. There is $B \subset \mathcal{T}_e$ non null such that $K':=\overline{\psi(B)}$ is compact (Lemma \ref{Lemma: Existence support on the canonical transverse}). So, if we have $g \in G$, $Z \in B$ and $gZ \in B$, then $\tau(g) \in K'^{-1}KK'$. Therefore, for all $ Y \in \Omega_X$ we have 
 $$\tau(T(Y,B)T(Y,B)^{-1}) \subset K'^{-1}KK'.$$ But $X$ is covered by finitely many translates of $\bigcup_Y T(Y,B)T(Y,B)^{-1}$ (Lemma \ref{Lemma: Key lemma hitting times}), so $\tau(X)$ is relatively compact.
 
 Suppose that we are given a measurable map $\psi: \Omega \cap \mathcal{T}_e \rightarrow \mathbb{H}(k)$ and a continuous group homomorphism $\pi: G \rightarrow \mathbb{H}(k)$ such that for all $Y \in \Omega \cap \mathcal{T}_e$ and $g \in T_h(Y, \Omega \cap \mathcal{T}_e)$ we have $\tau(g) = \psi(gY) ^{-1}\pi(g) \psi(Y)$ (corresponding to part (2) of Lemma \ref{Lemma: Consequence superrigidity}). We will prove that (2) of Theorem \ref{Theorem: Superrigidity for return times, version 2} must hold. We will follow the strategy of \cite{machado2020apphigherrank}. By Lemma \ref{Lemma: Existence support on the canonical transverse} there is $h \in \mathbb{H}(k)$ such that for all compact neighbourhoods $V$ of $h$ there is $B_V \subset \mathcal{T}_e$ non-null such that $\psi(B_V) \subset V$. Upon modifying $\psi$ we may assume without loss of generality that $h=e$. Now for any $Y_0 \in \mathcal{T}_e$ and any $Y \in \Omega_X$ we have that $YY_0^{-1}$ is uniformly discrete. Therefore, for any $V$ neighbourhood of $e$ we know by Lemma \ref{Lemma: Key lemma hitting times} that there is a finite set $F=F(Y_0,V)$ such that:
 $$Y_0 \subset FT_h(Y,B_V)T_h(Y,B_V)^{-1}.$$
Take $y \in Y_0$, then there is $f \in F$ such that $f^{-1}y \in T_h(Y,B_V)T_h(Y,B_V)^{-1}$. By Lemma \ref{Lemma: Miscellaneous hitting times sets},  there is $Y_1 \in B_V$ such that $f^{-1}yY_1 \in B_V$ as well. In particular, this yields
$$\tau(y) = \tau(f) \tau(f^{-1}y) =\tau(f)\psi(f^{-1}yY_1)\pi(f^{-1}y)\psi(Y_1) \in \tau(F)V\pi(F)^{-1}\pi(y)V.$$
Write $K_V = \tau(F)V\pi(F)^{-1}$ then we have the equality
$$\tau(y) \in K_V \pi(y) V$$
for all $y \in Y_0$.

 We have shown in \cite[\S 3.4]{machado2020apphigherrank} that if $\gamma \in G$ satisfies that: the projection of $\pi(Y_0 \cap Y_0\gamma^{-1})$ to $K \setminus H$ has dense closure in the (right-)visual boundary $\partial (K \setminus H)$  - where $K$ denotes a maximal compact subgroup - then $\tau(\gamma)=\pi(\gamma)$. Since $\pi$ is surjective according to our assumption on dimensions, it in fact suffices to show: (*) the projection of $Y_0 \cap Y_0\gamma^{-1}$ to $K \setminus G$ has dense closure in the (right-)visual boundary $\partial (K \setminus G)$  - where $K$ denotes a maximal compact subgroup of $G$ this time. We will show that this condition is fulfilled for many $\gamma \in G$ by invoking pointwise ergodic theorems. 
 
 Note first that the map
 \begin{align*}
 I_{\gamma}:  \Omega_X & \longrightarrow \Omega_X^{ext} \\
   Y & \longmapsto Y \cap Y\gamma^{-1}
 \end{align*} 
 is a well-defined $G$-equivariant map. Let us show that it is Borel when $\gamma \in X$. As a consequence of the uniform discreteness of $XX^{-1}$ there is a neighbourhood of the identity $W \subset G$ such that for all $Y \in \Omega_X$ and all $g \in G$ we have $|YX^{-1} \cap gW| \leq 1$. So 
 \begin{align*}
  |Y \cap gW| + |Y\gamma^{-1} \cap gW| - |Y \cap Y\gamma^{-1} \cap gW| &=  |Y \cup Y\gamma^{-1} \cap gW| \\
  &  \leq |YX^{-1} \cap gW| \\
  & \leq 1.
  \end{align*}
 Therefore, $ Y \cap Y\gamma^{-1} \cap gW$ is non-empty if and only if both $Y \cap gW$ and $Y\gamma^{-1} \cap gW$ are non-empty. In other words, 
 $$\{Y \in \Omega_X : Y \cap Y\gamma^{-1} \cap gW \neq \emptyset \} = U^{gW} \cap U^{gW\gamma}.$$
 So $I_{\gamma}$ is Borel (see e.g. \cite{machado2020apphigherrank}). As a consequence, the push-forward $\nu_{\gamma}$ of $\nu$ is a $G$-invariant ergodic measure. It is either proper or concentrated on $\{\emptyset\}$.
 
 Fix $D \subset G$ countable such that the sets of $\xi \in \partial\left(K \setminus G\right)$ such that $ x_0 \cdot d^n \rightarrow \xi$ is dense. Here,  $x_0$ denotes the class of $e$ in $(K \setminus G)$. Take $\gamma \in G$ and suppose that $\nu_{\gamma}$ is proper.  Then the semi-group generated by $d^{-1}$ acts ergodically on $(\Omega_X,\nu)$ according to the Howe--Moore property (e.g. \cite[2.2.20]{zimmer2013ergodic}. So for all $W \subset G$ open and relatively compact and almost all $Y \in \Omega_X$ we have that
 $$\lim_{n \rightarrow \infty} \frac{1}{n}\sum_{k=0}^n |d^{-k}(Y \cap Y\gamma^{-1}) \cap W | = \nu(U^W).$$
 In particular, there is a sequence of elements $(y_n)$ of $Y \cap Y\gamma^{-1}$ such that $d(y_n,e) \rightarrow \infty$ as $n$ goes to $\infty$ and $y_n \in d^{i_n}W$ for some $i_n \geq 0$. By considering the above for $W$ running over a countable basis of neighbourhoods of the identity, we find that $\gamma$ satisfies (*) for almost all $Y \in \Omega_X$. 
 
 Let us now show that there are many $\gamma \in G$ such that $\nu_{\gamma}$ is proper. Define $$S := \{ x \in X^{-1}X : \nu_x \text{ is proper }\}$$ and fix $W$ a relatively compact open neighbourhood of the identity. If $X^{-1}$ is not covered by finitely many right-translates of $S$, then we can find a sequence $(x_n)_{n \geq 0}$ of elements of $X^{-1}$ such that for any two $i \neq j$ $\nu_{x_ix_j^{-1}}$ is concentrated on $\{\emptyset\}$. In particular,
  $$\int_{\Omega_X} |Yx_j \cap Yx_i \cap W| d\nu(Y) = \int_{\Omega_X} |Y \cap Yx_ix_j^{-1} \cap Wx_j^{-1}|d\nu(Y) =  0$$
  for all $i,j \geq 0$. But for all $N \geq 0$ and $Y \in \Omega_X$ we have:
 
 \begin{equation}
 |YX^{-1} \cap W | \geq |\bigcup_{i=0}^N Yx_i \cap W| \geq \sum_{i=0}^N |Yx_i \cap W| - \sum_{0\leq i < j \leq N}|Yx_j \cap Yx_i \cap W|. \label{Eq: Massicot wagner invariant hull}
 \end{equation}
 Write $\mu_G$ the pull-back of $\nu$ via the periodization map. We know that $\mu_G$ is a Haar measure. In addition, since $XX^{-1}$ is uniformly discrete $|YX^{-1} \cap W | \leq C$ for some constant $C$. Integrating (\ref{Eq: Massicot wagner invariant hull}) against $\nu$ we find 
 
 $$ C \geq \sum_{i=0}^N \mu_G(Wx_i^{-1}) = (N+1)\mu_G(W)$$ 
 for all $N \geq 0$. A contradiction. So there is $F \subset X^{-1}$ finite such that $X^{-1} \subset SF$. Since $\tau_{|S} = \pi_{|S}$, we conclude that $S$ is the subset we are looking for.
 \end{proof}
 
\subsection{Finite generation and property (T)}

The goal of this section is to prove the following key result. 

\begin{proposition}\label{Proposition: Finite generation Lagarias-type theorem}
With $X$ as in Theorem \ref{Theorem: Lagarias-type theorem in simple groups}. The subgroup $\langle X\rangle$ is finitely generated. 
\end{proposition} 

To prove Proposition \ref{Proposition: Finite generation Lagarias-type theorem} we follow the same strategy as in \cite{machado2020apphigherrank} with the additional input of Lemma \ref{Lemma: restriction alpha-invariant vector to transverse} to help us restrict the cocycle we consider to the canonical transverse.
 
 \begin{proposition}\label{Proposition: Heredity Lagarias-type property (T)}
 Let $\langle X \rangle$ act on a separable Hilbert space $\mathcal{H}$ with almost invariant vectors. Then there is $\xi \in \mathcal{H}$ a unit vector such that $X \cdot \xi$ is relatively compact in the norm topology. 
 \end{proposition}
 
 \begin{proof}
 Following Zimmer's approach \cite{zimmer2013ergodic} (see also \cite{machado2020apphigherrank}), we can find a non-trivial $\phi: \Omega_X \rightarrow G$ such that for all $g \in G$ and $\nu$-almost all $Y  \in \Omega_X$ we have
 $$ \pi \left(\alpha_s(g,Y)\right) \left(\phi(Y)\right) = \phi(gY).$$
 
 In particular, $||\phi(Y)|| =||\phi(gY)||$. So $||\phi(Y)||$ is constant for almost all $Y$ by ergodicity of $\nu$. We may therefore assume that $\phi$ takes values in the set of unit vectors of $\mathcal{H}$. By Lemma \ref{Lemma: restriction alpha-invariant vector to transverse} there are thus $\Omega \subset \Omega_X$ that is $G$-invariant and of full measure and $\psi: \Omega \cap \mathcal{T}_e \rightarrow \mathcal{H}$ measurable such that 
 $$ \pi(g)\cdot \psi(Y) = \psi(gY)$$
 for all $Y \in \Omega \cap \mathcal{T}_e$ and $g\in T_h(Y, \Omega \cap \mathcal{T}_e)$. 
 
 Take $\xi \in \mathcal{H}$ as provided by Lemma \ref{Lemma: Existence support on the canonical transverse}. Then for every $\epsilon > 0$ the subset $$\mathcal{T}_{e,\epsilon}:=\psi^{-1}(B(\xi, \epsilon))$$ is non-null. According to Lemma \ref{Lemma: Key lemma hitting times}, 
 $$X \subset FT_h(Y, \mathcal{T}_{e,\epsilon})T_h(Y, \mathcal{T}_{e,\epsilon})^{-1}$$ for some finite subset $F$ and some $Y \in \Omega$. For all $g \in  T_h(Y, \mathcal{T}_{e,\epsilon})T_h(Y, \mathcal{T}_{e,\epsilon})^{-1}$ we can find $Y' \in \mathcal{T}_{e,\epsilon}$ such that $gY' \in \mathcal{T}_{e,\epsilon}$. So $$\psi(gY') = \pi\left(g\right)\left(\psi(Y')\right).$$ Since $\psi(Y'), \psi(gY') \in B(\xi, \epsilon)$, we have $\pi(g)(\xi) \in B(\xi, 2\epsilon)$. To sum up, we have $\pi(X)(\xi) \subset \pi(F)(B(\xi, 2\epsilon))$. Since there is such an $F$ for every $\epsilon$, and $\mathcal{H}$ is complete, $\pi(X)(\xi)$ is relatively compact.
 \end{proof}
 
 \begin{proof}[Proof of Proposition \ref{Proposition: Finite generation Lagarias-type theorem}]
 The proof of \cite[Prop. 13]{machado2020apphigherrank} holds almost verbatim. We reproduce it here for the sake of completeness.  Let $\mathcal{H}$ denote $\bigoplus_{\Delta \subset \langle X \rangle \text{ f.g.}} L^2(\langle X \rangle/\Delta)$. The sum of quasi-regular representations $\left(\pi, \mathcal{H}\right)$ almost has invariant vectors (e.g. \cite{MR2415834}). So we can find $\phi \in \mathcal{H}$ such that $X \cdot \phi$ has compact closure in the norm topology. Upon projecting $\phi$ to a well-chosen factor of $\mathcal{H}$, we may assume that $\phi \in  L^2(\langle X \rangle/\Delta)$ for some finitely generated subgroup $\Delta$ of $\langle X \rangle$. Now, define $$X(\delta, \phi):=\{x \in X^{-1}X : ||x\cdot \phi - \phi|| \leq \delta\}.$$
 Then $(X(\delta, \phi))_{\delta > 0}$ is a family of subsets such that $X$ is covered by finitely many translates of $X(\delta, \phi)$ and $X(\delta, \phi)$ is contained in $X^{-1}X$ (see \cite{machado2020apphigherrank}). Now let $p: \langle X\rangle \rightarrow \langle X\rangle/\Delta$ denote the natural projection. Take $\gamma \in \langle X \rangle$ such that $\phi(p(\gamma))=\alpha > 0$. So for all $x \in X(\alpha/2,\phi)$ we have 
  $$| \phi(p(x^{-1}\gamma))- \phi(p(\gamma))| \leq ||\pi(x)(\phi)-\phi ||< \alpha/2,$$
 meaning $x^{-1}\gamma \in \phi^{-1}([\alpha/2; +\infty))$. Since $\phi^{-1}([\alpha/2; +\infty))$ is finite, we can find a finite set $F$ of representatives of $\phi^{-1}([\alpha/2; +\infty))$ in $\langle X \rangle$. Then $x^{-1} \gamma \Delta \cap F \Delta \neq \emptyset$ and $X(\alpha/2,\phi)$ is contained in $F\Delta\gamma^{-1}$. But there is a finite subset $F' \subset \langle \Lambda \rangle$ such that $X \subset F'X(\alpha/2, \phi) \subset F'F\Delta\gamma^{-1}$. Since $$X \subset F' \cup F \cup \Delta \cup \{\gamma^{-1}\} \subset \langle X \rangle,$$
  $F' \cup F \cup \Delta \cup \{\gamma^{-1}\}$ generates $\langle X \rangle$. But $\Delta$ is finitely generated. So $\langle X \rangle$ is finitely generated.
 \end{proof}

\subsection{Proof of Theorem \ref{Theorem: Lagarias-type theorem in simple groups}}

\begin{proof}
Suppose as we may that $G$ is a closed subgroup of   $\GL_n(k)$ for some $n$. According to Proposition \ref{Proposition: Finite generation Lagarias-type theorem}  the subgroup $\langle X \rangle$ is finitely generated. So the field $K$ generated by entries of elements of $\langle X \rangle$ is finitely generated. Let $S \subset XX^{-1}$ be the subset provided by Theorem \ref{Theorem: Superrigidity for return times, version 2}. We will show that $S^m$ is uniformly discrete for all $m \geq 0$. Let $W$ be a compact neighbourhood of the identity and take an integer $m \geq 0$. If $S^{2m} \cap W$ is infinite, then the set $E$ of entries of elements in $S^{2m} \cap W$ is infinite as well. By a result of Breuillard--Gelander \cite{MR1966638}, there is a local field $l$ and a field embedding $\sigma: K \rightarrow l$ such that $\sigma(E)$ is unbounded. Denote by $\sigma_0$ the natural group homomorphism $\GL_m(K) \rightarrow \GL_m(l)$ given by applying $\sigma$ entry-wise. Then $\sigma_0(S^{2m} \cap W)$ is unbounded. As $$\sigma_0(S^{2m} \cap W) \subset \sigma_0\left((XX^{-1})^{2m}\right),$$
the subset $\sigma_0(X)$ must be unbounded. Moreover, the Zariski-closure of $\sigma_0(\langle X \rangle)$ is semi-simple and has dimension at most $\dim G$. By Theorem \ref{Theorem: Superrigidity for return times, version 2}, there is a continuous group homomorphism $\pi: G \rightarrow \GL_m(l)$ such that $\left(\sigma_0\right)_{|\langle S \rangle} = \pi_{|\langle S \rangle}$. But
$$\sigma_0\left(S^{2m} \cap W \right) = \pi\left(S^{2m} \cap W \right) \subset \pi(W).$$
So $\sigma_0\left(S^{2m} \cap W \right)$ is bounded. A contradiction. 

Therefore $S^{2m} \cap W$ is finite for all $m \geq 0$. So $S^m$ is uniformly discrete for all $m \geq 0$. But there is $F$ finite such that $X \subset FS$. By Lemma \ref{Lemma: finite co-volume passes to commensurability} there is thus a Borel subset $\mathcal{F}$ of finite Haar measure such that $S^2 \mathcal{F} = G$. Since $S^5$ is uniformly discrete, we find that $S^2$ is an approximate subgroup as a consequence of Lemma \ref{Lemma: Hull and commensurability}. Hence, $S^2$ is an approximate lattice. 

\end{proof}

\begin{remark}
This last part of the argument is the only one that is not quantitative. We do not know if this argument can be made quantitative as we heavily rely on the superrigidity theorem which is, by essence, global.
\end{remark}

\section{Acknowledgement}
I am deeply indebted to Tobias Hartnick for inspiring discussions during a stay at the Karlsruhe Institute of Technology. I want to thank Matthew Tointon for bringing to my attention Konieczny's paper and the discussion that followed. I thank Michael Bj\"{o}rklund, Tobias Hartnick and Yakov Karasik for communicating their result. 


\end{document}